%% file: main.tex
\documentclass[12pt]{article}
\input{packages}
\input{environnements}

\input{commandes.tex}

\title{On dissociated infinite permutation groups}

\author{Rémi BARRITAULT, Colin JAHEL, Matthieu JOSEPH}
\date{}

\begin{document}
\maketitle
\begin{abstract}
    The goal of this paper is threefold. First, we describe the notion of dissociation for closed subgroups of the group of permutations on a countably infinite set and explain its numerous consequences on unitary representations (classification of unitary representations, Property (T), the Howe-Moore property, etc.) and on ergodic actions (non-existence of type III non-singular actions, Stabilizer rigidity, etc.). Some of the results presented here are new, others were proved in different contexts (notably some results of Tsankov). Second, we introduce a new method to prove dissociation. It is based on a reinforcement of the classical notion of strong amalgamation, where we allow to amalgamate over countable sets. Third, we apply this technique of amalgamation to provide new examples of dissociated closed permutation groups, including isometry groups of some countable metrically homogeneous spaces, automorphism groups of diversities, and more.  
\end{abstract}

\noindent\textbf{MSC:}  Primary: 22A25, 22F50. Secondary: 37A15, 60G09.\\
\textbf{Keywords: } Infinite permutation groups, Unitary representations, Type I groups, Property (T), The Howe-Moore property, Stabilizer rigidity, de Finetti Theorem.

\tableofcontents

\section*{Introduction}

In this paper, we gather together under a common name two techniques -- \emph{independence and orthogonality} -- that were used many times in the past in the study of infinite permutation groups. \emph{Dissociation}, which is the main notion of the present paper, finds its roots in exchangeability theory. In that context, dissociation of a random process corresponds to \emph{independence}. It is a central notion designed to generalize the famous de Finetti's theorem and has been extensively studied especially in the setting of random arrays \cite{Aldous,KallenbergSymmetries}. In the last decades, this notion of dissociation has percolated in a domain at the intersection of probability and model theory, which studies random structures that are invariant under the automorphism group of some structure (see e.g.\ \cite{CraneTowsner, DiaconisJanson, JT}).

Another technique that has been widely used is that of \emph{orthogonality}, that leads to classification results of unitary representations for many automorphism groups of structures \cite{BJJ24, liebermann, Tsankov}. One of the aims of this paper is to explain how these techniques of independence and orthogonality are one and the same, which we call dissociation.

An important motivation for introducing this notion is the fact that this phenomenon has already been observed for the class of $\aleph_0$-categorical structures in a weaker form \cite[Prop.~3.2]{JT}. This observation has consequences, some ergodic theoretic \cite{jahel2023stabilizers,JT, tsankov2024nonsingularprobabilitymeasurepreservingactions}, others on unitary representations \cite{Tsankov}. Dissociation beyond the $\aleph_0$-categorical case was already applied to the Urysohn spaces by the authors \cite{BJJ24}. We now intend to start a more systematic study of this property.

\begin{notation}
    In this paper, $\Omega$ stands for a countably infinite set and $\Sym(\Omega)$ for the group of all permutations of $\Omega$, equipped with the topology of pointwise convergence. A closed permutation group (on $\Omega$) is a closed subgroup $G$ of $\Sym(\Omega)$. For every $A\subseteq \Omega$, we denote by $G_A$ the \textit{pointwise stabilizer} of $A$ and by $\Nm G A$ the \emph{setwise stabilizer} of $A$, that is:
    $$G_A \coloneqq \{g \in G, \ \forall a \in A, \ g(a)=a\} \quad \text{and} \quad  \Nm G A \coloneqq \{g\in G, \ g(A)=A\}.$$
    We emphasize the fact that the data of a permutation group does not reduce to that of a topological group. It is a specific presentation as a subgroup of $\Sym(\Omega)$ and, in particular, $\Omega$ itself is part of the data.
\end{notation}

\paragraph{Dissociation for unitary representations.} Let $G$ be a closed permutation group on $\Omega$. A \emph{unitary representation} of $G$ on a Hilbert space $\HH$ is a group homomorphism $\pi\colon G\to\UU(\HH)$ such that for all $\xi\in\HH$, the map $g\mapsto \pi(g)\xi$ is continuous. In other words, we require $\pi$ to be continuous for the strong operator topology on the unitary group $\mathcal{U}(\mathcal{H})$.

For every finite subset $A\subseteq\Omega$, we will denote by 
$$\HH_A\coloneqq \{\xi\in \HH, \ \forall g\in G_A, \ \pi(g)\xi = \xi\} $$ the closed subspace of $G_A$-invariant vectors and write $p_A$ for the orthogonal projection onto $\HH_A$. One key fact about unitary representations of permutation groups is that invariant vectors determine the whole representation (see e.g.\ \cite[Lem.~3.1]{Tsankov}), that is to say:
$$\HH = \overline{\bigcup\{\HH_{A}, \ A\subseteq \Omega \text{ finite}\}}.$$
Dissociation is about how the invariant spaces are arranged with respect to each other.

\begin{definition*}
The unitary representation $\pi$ is \defin{dissociated} if for all $A,B\subseteq\Omega$ finite, we have $p_Ap_B=p_{A\cap B}$. Geometrically, this means that the subspaces $\HH_A\cap (\HH_{A\cap B})^\bot$ and $\HH_B\cap (\HH_{A\cap B})^\bot$ are orthogonal. 
\end{definition*}

One notable feature that we will use in the paper is that there is no real need to reduce our attention to finite subsets of $\Omega$. Indeed, for any subset $A \subseteq \Omega$, one can define:
$$\HH_A \coloneqq \overline{\bigcup\{\HH_{B}, \ B\subseteq A \text{ finite}\}}$$
and we still denote by $p_A$ the orthogonal projection onto $\HH_A$. Dissociation is then equivalent to $p_Ap_B = p_{A \cap B}$ for all $A, B \subseteq \Omega$, not necessarily finite.

\paragraph{Dissociation for Boolean p.m.p.\ actions.} Let $G$ be a closed permutation group on $\Omega$. A Boolean p.m.p.\ action of $G$ on a standard probability space $(X,\mu)$ is a homomorphism $\alpha\colon G\to\Aut(X,\mu)$ which is continuous for the weak topology on $\Aut(X,\mu)$. For every finite subset $A\subseteq\Omega$, let $\FF_A$ be the $\sigma$-algebra of $G_A$-invariant measurable subsets, i.e.:
$$\FF_A\coloneqq \{Y \subseteq X,\ Y \text{ is measurable and } \forall g\in G_A,\ \mu(g Y\triangle Y)=0\}.$$
The Boolean p.m.p.\ action $\alpha$ is \defin{dissociated} if for every finite subset $A,B\subseteq\Omega$, the $\sigma$-algebras $\FF_A$ and $\FF_B$ are independent conditionally on $\FF_{A\cap B}$.

\paragraph{Dissociated groups.} In Section \ref{sec.definition.dissociation}, we discuss the relation between dissociation for unitary representations and for Boolean p.m.p.\ actions. We prove that such an action is dissociated if and only if its Koopman representation is dissociated. We then prove that for a closed permutation group $G$, every unitary representation is dissociated if and only if every Boolean p.m.p.\ action is dissociated. If this holds, we say that $G$ is \defin{dissociated}. 

\paragraph{Compendium on dissociated groups.} One of the goals of this paper is to collect in one place some results related to dissociation that were either proved in other papers, e.g.\ \cite{BJJ24, jahel2023stabilizers, tsankov2024nonsingularprobabilitymeasurepreservingactions, Tsankov}, or are proved in the present paper.

\begin{theorem*}\label{thm.intro.main}
    Let $G$ be a closed permutation group on $\Omega$. If $G$ is dissociated, then the following statements hold. 

    \begin{enumerate}[leftmargin=20px, label=\Roman*.]

        \item Unitary representations:
        \begin{enumerate}[leftmargin=15px, label=\arabic*)]
            \item \emph{\textbf{Rigidity of unitary representations} (Corollary \ref{cor.classification.dissociated.unirep}).} Every irreducible unitary representation of $G$ is induced from an irreducible representation of the setwise stabilizer $G_{\{A\}}\coloneqq\{g\in G\colon g(A)=A\}$ for some finite  $A\subseteq\Omega$, which is trivial on the pointwise stabilizer $G_A$. Moreover, every unitary representation of $G$ is a direct sum of irreducible ones. 
        \item \emph{\textbf{Property (T)} (Theorem \ref{thm.dissociated.property(T)}).} $G$ has Kazhdan's Property (T).
        
          \item \emph{\textbf{The Howe-Moore property} (Theorem \ref{thm.Howe.Moore.dissociated}).} If every proper open subgroup of $G$ is coarsely bounded, then $G$ has the Howe-Moore property.
        \end{enumerate}
        \item Ergodic theory:
        \begin{enumerate}[leftmargin=15px, label=\arabic*)]
            \item \emph{\textbf{Rigidity of non-singular actions} (Theorem \ref{thm.nonsingular.induced}).} Every non-singular ergodic action of $G$ is induced by a probability measure-preserving action of the setwise stabilizer $G_{\{A\}}$ for some finite subset $A\subseteq\Omega$. Moreover, every non-singular action of $G$ is a disjoint union of ergodic ones. 
        \item\emph{\textbf{Stabilizer rigidity} (Theorem \ref{thm.pmprigiddissociatedexpanded}).} If $G$ is a proper subgroup of $\Sym(\Omega)$ which acts transitively on $\Omega$, then any Borel p.m.p.\ ergodic action of $G$ is either essentially free or essentially transitive. 
    \item\emph{\textbf{Rigidity of invariant random processes} (Corollary \ref{cor.deFinetti}).} If $G$ acts transitively on $\Omega$, then every random process $(X_\omega)_{\omega\in\Omega}$ whose law is $G$-invariant and ergodic is i.i.d.
        \end{enumerate}
    \end{enumerate}
\end{theorem*}

\paragraph{How to obtain dissociation?} To tackle this question, we change our point of view in Section \ref{sec.uniform.non.algebraicity} and consider closed subgroups of $\Sym(\Omega)$ as groups of automorphisms of countable relational structures and present two general methods for obtaining dissociation. First, through \textit{approximating sequences}, when the automorphism group can be seen as a coherent limit of dissociated groups. Second, by introducing the notion of \textit{strong cofinite amalgamation over countable sets} (abbrev.\ $\sigma$-SAP, see Definition \ref{defn:SAPAS}). Our main new result in this section is Theorem \ref{thm:omega-SAP.dissociated}, which proves that if a ultrahomogeneous structure satisfies $\sigma$-SAP and weakly eliminates imaginaries then its automorphism group is dissociated. The method we use to prove this result was developed by Tsankov to obtain de Finetti-like results \cite[Sec.~4]{tsankov2024nonsingularprobabilitymeasurepreservingactions}.

\paragraph{New examples of dissociated groups.} Finally, we apply in Section \ref{sec.examples} the methods from Section \ref{sec.uniform.non.algebraicity} to provide new examples of dissociated groups, such as isometry groups of some metrically homogeneous spaces, automorphism groups of some diversities and more. In particular, we obtain examples going beyond the more widely studied Roelcke-precompact case, some being locally Roelcke-precompact, others being coarsely bounded but not Roelcke-precompact.

\paragraph{Acknowledgments.} We thank Todor Tsankov for sharing a first version of \cite{tsankov2024nonsingularprobabilitymeasurepreservingactions} which inspired many results in the present paper. A few results in the present paper (e.g.\ Theorem \ref{thm.nonsingular.induced}) are even direct adaptations of his results. We also thank Yves Benoist and Christian Rosendal for pointing out to us the similarities between the notion of dissociation and Mackey's work on systems of imprimitivity.

\section{About the definition of dissociation}\label{sec.definition.dissociation}

\subsection{Dissociation: unitary representations and Boolean p.m.p.\ actions}

Let $(X,\mu)$ be a standard probability space and $\FF_1,\FF_2,\FF_3$ be three $\sigma$-algebras of measurable sets on $X$ satisfying $\FF_2\subseteq\FF_1\cap\FF_3$. We say that $\FF_1$ and $\FF_3$ are \emph{independent conditionally on $\FF_2$} if the following holds: for every $\FF_3$-measurable function $f\in\L^2(X,\mu)$,  we have $\E[f\mid \FF_1]=\E[f\mid \FF_2]$, where $\E[\cdot\mid\FF_i]\colon \L^2(X,\mu)\to\L^2(X,\FF_i,\mu)$ is the conditional expectation.

Let $\HH$ be a Hilbert space and $\HH_1,\HH_2,\HH_3$ be three closed subspaces satisfying $\HH_2\subseteq\HH_1\cap\HH_3$ and let $p_1,p_2,p_3$ be the associated orthogonal projections. We say that $\HH_1$ and $\HH_3$ are \emph{orthogonal conditionally on $\HH_2$}, and write:
$$\HH_1 \bot_{\HH_2} \HH_3 $$
if $p_1p_3= p_2$, i.e.\ if the subspaces $\HH_1\cap\HH_2^\bot$ and $\HH_3\cap\HH_2^\bot$ are orthogonal. Note that if $\HH_1 \bot_{\HH_2} \HH_3 $ then $\HH_2 = \HH_1 \cap \HH_3$.

These two notions of conditional independence/orthogonality are related by the following classical result. 
 
\begin{lemma}\label{lem.equivalence.independence.orthogonality}
    Let $(X,\mu)$ be a standard probability space and $\FF_1,\FF_2,\FF_3$ be three sub-$\sigma$-algebras on $X$ satisfying $\FF_2\subseteq\FF_1\cap\FF_3$. If $L^2(X,\FF_i,\mu)$ denotes the space of square-integrable, $\FF_i$-measurable complex valued functions on $(X,\mu)$ (for $i=1,\ 2,\ 3$), then the following are equivalent. 
    \begin{enumerate}[label=(\roman*)]
        \item $\FF_1$ and $\FF_3$ are independent conditionally on $\FF_2$.
        \item $\L^2(X,\FF_1,\mu)$ and $\L^2(X,\FF_3,\mu)$ are orthogonal conditionally on $\L^2(X,\FF_2,\mu)$. 
    \end{enumerate}
\end{lemma}

\begin{proof}
Write $\HH_i\coloneqq\L^2(X,\FF_i,\mu)$ and $p_i\coloneqq \E[\cdot \mid \FF_i] \colon \L^2(X,\mu) \to \HH_i$ for $i=1,\ 2,\ 3$. Notice that $\FF_2\subseteq\FF_1\cap\FF_3$ implies $\HH_2\subseteq\HH_1\cap\HH_3$. Since $p_i$ is the orthogonal projection onto $\HH_i$, the following equivalences hold:
\begin{align*}
    \FF_1\indep_{\FF_2}\FF_3 
    &\Leftrightarrow \forall f\in\HH_3,\  p_1f=p_2f \\
    &\Leftrightarrow p_1p_3 =p_2\\
    &\Leftrightarrow \HH_1 \bot_{\HH_2} \HH_3. 
    \qedhere
\end{align*}
\end{proof}

Let $(X,\mu)$ be a probability space and $\Aut(X,\mu)$ be the group of measure-preserving bijections of $(X,\mu)$, identified up to measure zero. It is equipped with the weak topology, defined as follows: $T_n\to T$ if for all measurable $A\subseteq X$, $\mu(T_n(A)\triangle T(A))\to 0$.  

A Boolean p.m.p.\ action $\alpha\colon G\to\Aut(X,\mu)$ of a topological group $G$ on $(X,\mu)$ is ergodic if $\FF_\emptyset$ is trivial (i.e.\ if $\forall Y\in\FF_\emptyset,\ \mu(Y)\in\{0,1\}$). The following lemma connects $G$-invariant functions with $\FF_\emptyset$-measurable functions. While often stated for Borel p.m.p.\ actions of locally compact groups, its proof readily adapts to Boolean p.m.p\ actions of topological groups.

\begin{lemma}\label{lem.invariant.functions}
    Let $G$ be a topological group, let $\alpha\colon G\to\Aut(X,\mu)$ be a Boolean p.m.p.\ action. A function $f\colon X\to\C$ is $\FF_\emptyset$-measurable if and only if for every $g\in G$, we have $f\circ g=f\ \mu$-almost everywhere.
\end{lemma}

As a consequence, we connect the two notions of dissociation -- for unitary representations and Boolean p.m.p.\ actions -- defined in the introduction. First, recall that a Boolean p.m.p.\ action $\alpha \colon G\to\Aut(X,\mu)$ of a topological group $G$ on a standard probability space $(X,\mu)$ naturally defines a unitary representation of $G$ on $\L^2(X,\mu)$, called the \emph{Koopman representation of $\alpha$}, as follows:
$$\forall f\in \L^2(X,\mu), \forall g \in G, \forall x \in X, \ (g\cdot f)(x) = f(\alpha(g^{-1})x).$$

\begin{lemma}\label{lem.koopman.dissociated}
    Let $G\leq\Sym(\Omega)$ be a closed permutation group and $\alpha \colon G\to\Aut(X,\mu)$ be a Boolean p.m.p.\ action. Then $\alpha$ is dissociated if and only if its Koopman representation is dissociated.
\end{lemma}

\begin{proof}
    Let $\HH\coloneqq\L^2(X,\mu)$. For every finite subset $C\subseteq\Omega$, Lemma \ref{lem.invariant.functions} implies that $\HH_C=\L^2(X,\FF_C,\mu)$. Therefore, Lemma \ref{lem.equivalence.independence.orthogonality} shows that $\alpha$ is dissociated if and only if $\kappa$ is. 
\end{proof}

The following result is standard using Gaussian actions. For a proof in the context of Boolean p.m.p.\ actions of Polish groups, we refer to Section 8.2 of \cite{CarderiGiraudLeMaitre} and the references therein. 

\begin{theorem}\label{thm.Gaussian.action}
    Let $G$ be a Polish group and $\pi \colon G\to\UU(\HH)$ be a unitary representation. There exists a standard probability space $(X,\mu)$ and a Boolean p.m.p.\ action $G\to\Aut(X,\mu)$ whose (complex) Koopman representation contains $\pi$ as a subrepresentation.   
\end{theorem}

We are now ready to prove the following result. 

\begin{theorem}\label{thm.intro.equivalence.dissociation}
    Let $G$ be a closed permutation group. The following are equivalent. 
    \begin{enumerate}[label=(\roman*)]
        \item\label{item.dissociation.unirep} Every unitary representation of $G$ is dissociated.
        \item\label{item.dissociation.Booleanpmp} Every Boolean p.m.p.\ action of $G$ is dissociated. 
    \end{enumerate}
    If this holds, we say that $G$ is dissociated. 
\end{theorem}

\begin{proof}
    Assume that every unitary representation of $G$ is dissociated. Let $\alpha \colon G\to\Aut(X,\mu)$ be a Boolean p.m.p.\ action. By assumption, its Koopman representation is dissociated. Therefore, $\alpha$ is dissociated by Lemma \ref{lem.koopman.dissociated}

    Assume that every Boolean p.m.p.\ action of $G$ is dissociated. Let $\pi \colon G\to\UU(\HH)$ be a unitary representation. By Theorem \ref{thm.Gaussian.action}, there exists a Boolean p.m.p.\ action $\alpha\colon G\to\Aut(X,\mu)$ whose Koopman representation $\kappa\colon G\to\UU(\L^2(X,\mu))$ contains $\pi$ as a subrepresentation. Since $\alpha$ is dissociated, so is $\kappa$ by Lemma \ref{lem.koopman.dissociated}. But dissociation passes to subrepresentations \cite[Lem.~3.4]{BJJ24} so $\pi$ is also dissociated. 
\end{proof}

\subsection{Lattice of open subgroups and topological simplicity}

Let $G\leq\Sym(\Omega)$ be a closed permutation group. We say that $G$ \mkbibemph{has no algebraicity} if for every finite subset $A\subseteq\Omega$, the orbits of the pointwise stabilizer $G_A$ on $\Omega\setminus A$ are all infinite. We say that $G$ \emph{weakly eliminates imaginaries} if for every open subgroup $V\leq G$, there exists a finite subset $A\subseteq\Omega$ such that $G_A\leq V$ and $[V:G_A]<+\infty$. The combination of these two properties leads to a precise understanding of the lattice of open subgroups, as explained in the following lemma. 

\begin{lemma}\label{lem.caracterisation.noAlg.WEI}
    Let $G$ be a closed subgroup of $\Sym(\Omega)$. Assume that $G$ acts without fixed point on $\Omega$. Then the following are equivalent. 
\begin{enumerate}[label=(\roman*)]
    \item\label{item.noAlgWEI.1} $G$ has no algebraicity and weakly eliminates imaginaries.
    \item\label{item.noAlgWEI.2} For all finite subsets $A,B\subseteq\Omega$, the subgroup $\langle G_A,G_B\rangle$ generated by $G_A$ and $G_B$ is equal to $G_{A\cap B}$. 
    \item\label{item.noAlgWEI.3} For every open subgroup $V\leq G$, there exists a unique finite subset $A\subseteq\Omega$ such that $G_A\leq V\leq G_{\{A\}}$. 
\end{enumerate}
\end{lemma}
\begin{proof}
    The equivalence between \ref{item.noAlgWEI.1} and \ref{item.noAlgWEI.2} is Lemma 3.5 of \cite{jahel2023stabilizers} while \ref{item.noAlgWEI.3} is simply an efficient reformulation of the above definitions. Indeed, assume $G$ satisfies \ref{item.noAlgWEI.1} and let $V\leq G$ be an open subgroup. By weak elimination of imaginaries, there exists $A\subseteq \Omega$ finite such that $G_A \leq V$ with finite index. Necessarily, every element of $V\cdot A$ has a finite orbit under the action of $G_A$. Since $G$ has no algebraicity, $V\cdot A \subseteq A$ i.e.\ $V\leq \Nm G A$. 
    
    Assume now $B$ is another finite subset of $\Omega$ such that $G_B \leq V \leq \Nm G B$. Then $G_A\leq V \leq \Nm G B$ and $G_A$ stabilizes $B$. In particular, every element of $B$ has a finite orbit under the action of $G_A$. By the no algebraicity hypothesis, $B \subseteq A$. The reverse inclusion also holds by symmetry.

    Conversely, assume \ref{item.noAlgWEI.3}. Clearly, $G$ has weak elimination of imaginaries. Now let $A \subseteq \Omega$ be finite and assume $G_A \cdot a$ is finite for some $a \in \Omega$. Then $B \coloneq A \cup G_A \cdot a$ is finite and stable under the action of $G_A$. Thus:
    $$G_B \leq G_A \leq \Nm G B.$$
    But we obviously also have $G_A \leq G_A \leq \Nm G A$. By uniqueness in \ref{item.noAlgWEI.3}, $A=B$ i.e.\ $a\in A$ and $G$ has no algebraicity.
\end{proof}


In fact, our definition of dissociation is tailored for permutation groups that have no algebraicity and weakly eliminate imaginaries. In particular, we have the following lemma, which appears as Remark 3.2 of \cite{BJJ24}.

\begin{lemma}\label{lem.dissociated.noAlg.WEI}
    Let $G\leq\Sym(\Omega)$ be a closed permutation group. Assume that $G$ acts without fixed points on $\Omega$. If $G$ is dissociated, then $G$ has no algebraicity and eliminates weakly imaginaries. 
\end{lemma}

The proof of the following result was explained to us by David Evans (personal communication).

\begin{proposition}
    Let $G$ be a transitive, closed subgroup of $\Sym(\Omega)$. If $G$ has no algebraicity and weakly eliminates imaginaries, then $G$ is topologically simple.
\end{proposition}

\begin{proof}
Let $H$ be a closed, normal, proper subgroup of $G$. Expressing that $H$ is proper and closed, and therefore not dense in $G$, we get an integer $n\geq 1$ and a $G$-orbit $\Delta \subseteq \Omega^n$ such that $H$ is not transitive on $\Delta$. As $H\trianglelefteq G$, the $H$-orbits form a $G$-equivariant equivalence relation $\sim$ on $\Delta$. Fix $\x\in\Delta$ and write $[\x]_\sim$ for the equivalence class $H\cdot \x$ of $\x$. It is clear that $H$ is contained in the subgroup $V\coloneqq \{g\in G\colon g([\x]_\sim)=[\x]_\sim\}$ which is open since it contains the pointwise stabilizer $G_{\x}$. By Lemma \ref{lem.caracterisation.noAlg.WEI}, there exists a unique finite subset $A\subseteq\Omega$ such that $G_A \leq V\leq G_{\{A\}}$. Necessarily, $A\neq \emptyset$ otherwise $[\x]_\sim = \Delta$, a contradiction. Now, $H\leq G_{\{A\}}$ and, by normality, we get:
$$H\leq K\coloneqq\bigcap_{g\in G}G_{\{g(A)\}}.$$
Let us prove that $K$ is trivial. Since $K$ is normal in $G$ and $G$ acts transitively on $\Omega$, it suffices to prove that $K$ has a fixed point on $\Omega$. Fix any $a \in A$. Since the $G_a$-orbits on $\Omega\setminus\{a\}$ are infinite, there exists by Neumann's lemma \cite[Thm.~1]{Birch_1976} an element $g_0\in G_a$ such that $A\cap g_0A=\{a\}$. Thus $K\leq G_{\{A\}}\cap G_{\{g_0(A)\}}\leq G_a$, which concludes the proof.   
\end{proof}

\begin{corollary}
    Let $G\leq\Sym(\Omega)$ be a transitive, closed permutation group. If $G$ is dissociated, then $G$ is topologically simple.
\end{corollary}

\section{Unitary representations of dissociated groups}

\subsection{Classification of unitary representations}

In order to state the classification result of dissociated unitary representations that we obtained in \cite{BJJ24}, let us first recall the notion of induced representations. Let $G$ be a separable topological group and let $H\leq G$ be an open subgroup. Then $G/H$ is at most countable by separability of $G$ and we endow it with the counting measure. Let $\sigma \colon H\to\UU(\KK)$ be a unitary representation of $H$. Let $\EE$ be the space of maps $f\colon G\to\KK$ such that for every $g\in G$ and $h\in H$ we have $f(gh)=\sigma(h\inv)f(g)$. Notice that for all $f,f_1,f_2 \in \EE$, the maps $g\mapsto \left<f_1(g),f_2(g) \right>$ and $g\mapsto\lVert f(g)\rVert$ are constant on each left $H$-coset. Denote by $\left<f_1,f_2\right>(q)$ and $\lVert f(q)\rVert$ their respective value on the coset $q\in G/H$. 
Let $\HH$ be the Hilbert space of all $f\in \EE$ such that $\sum_{q\in G/H}\lVert f(q)\rVert^2 <+\infty$
with inner product given by \[\left<f_1,f_2\right> = \sum_{q\in G/H}\left<f_1,f_2\right>(q).\] 
The \textit{induced representation} $\pi\coloneqq\mathrm{Ind}_H^G(\sigma)$ is the representation of $G$ on $\HH$ defined by 
\[\pi(g)f \colon x\mapsto f(g\inv x)\quad\text{ for all }g\in G, f\in\HH.\]
Since $H$ is open, $\Ind_H^G(\sigma) \colon G\to\UU(\HH)$ is indeed continuous. 

Induction is a powerful technique whose importance was promoted by Mackey in his seminal work on representation theory of locally compact groups. It turns out that dissociation can be phrased in terms of systems of imprimitivity, a key notion in Mackey's theory. In what follows, if $\HH$ is a Hilbert space, $\BB(\HH)$ will denote the space of bounded operators on $\HH$. Let $G$ be a topological group, $(X,\AA)$ be a standard Borel space and $G\times X\to X$ be a Borel action. A \emph{system of imprimitivity} is a couple $(\pi,p)$, where $\pi\colon G\to\UU(\HH)$ is a unitary representation and $p\colon\AA\to\BB(\HH)$ is a projection-valued measure (see Section 2.5 of \cite{Mackey}) such that for every $A\in\AA$ and $g\in G$, 
\begin{equation}\label{eq.equivariance}
\pi(g)p(A)\pi(g\inv)=p(gA).
\end{equation}

For a closed permutation group $G\leq\Sym(\Omega)$, there is a natural Borel action, that of $G$ on $\Omega$. Moreover, given a unitary representation $\pi  \colon G\to\UU(\HH)$, there is a natural projection-valued measure $p_\pi\colon\PP(\Omega)\to\BB(\HH)$, by setting $p_\pi(A) = p_A$. Notice that equation \eqref{eq.equivariance} always holds in that case. Moreover, dissociation easily rephrases as follows. 

\begin{fact}
    A unitary representation $\pi \colon G\to\UU(\HH)$ with no non-zero invariant vector of a closed permutation group $G\leq\Sym(\Omega)$ is dissociated if and only if $(\pi, p_\pi)$ is a system of imprimitivity. 
\end{fact}   

In \cite[Thm.~3.9]{BJJ24}, we revisit Mackey's famous Imprimitivity Theorem in the context of closed permutation groups. Below, if $G$ is a group and $H_1, H_2$ are two conjugate subgroups of $G$, say $H_1 = gH_2g^{-1}$ for some $g \in G$, and if $\sigma$ is a representation of $H_1$, $\sigma^g$ will denote the representation $h\mapsto \sigma(ghg^{-1})$ of $H_2$ on the same Hilbert space as $\sigma$.

\begin{theorem}
        Let $G\leq\Sym(\Omega)$ be a closed permutation group without algebraicity.
    \begin{enumerate}
        \item The dissociated irreducible unitary representations of $G$ are exactly the unitary representations isomorphic to one of the form $\Ind_{G_{\{A\}}}^G(\sigma)$ where $A$ ranges over the finite subsets of $\Omega$ and $\sigma$ over the irreducible representations of the finite group $G_{\{A\}}/G_A$.

        \item Two such irreducible representations $\Ind_{G_{\{A\}}}^G(\sigma)$ and $\Ind_{G_{\{ B\}}}^G(\tau)$ are isomorphic if and only if there exists $g\in G$ such that $g(B) = A$ and $\sigma^g \simeq \tau$.

        \item Every dissociated unitary representation of $G$ splits as direct a sum of irreducible subrepresentations.
    \end{enumerate}
\end{theorem}

In particular, we get the following classification of unitary representations for dissociated groups. 

\begin{corollary}\label{cor.classification.dissociated.unirep}
   Let $G\leq \Sym(\Omega)$ be a closed permutation group. If $G$ is dissociated, then every irreducible unitary representation of $G$ is induced from an irreducible representation of the setwise stabilizer $G_{\{A\}}\coloneqq\{g\in G\colon g(A)=A\}$ for some finite  $A\subseteq\Omega$, which is trivial on the pointwise stabilizer $G_A$. Moreover, every unitary representation of $G$ is a direct sum of irreducible ones.   
\end{corollary}
If $E\subseteq \BB(\HH)$ for some Hilbert space $\HH$, we will denote by $E'\coloneqq \{h\in \BB(\HH), \ \forall k\in E, \ h\circ k = k\circ h\}$ the \emph{commutant of $E$}. Recall that a unitary representation $\pi \colon G\to\UU(\HH)$ is of type I if the von Neumann algebra $\pi(G)''$ it generates is of type I. A group is of type I if all of its unitary representations are of type I. We refer to \cite[Sec.~6]{BdlH} for a modern reference on type I topological groups. Here is a direct corollary of the above classification of unitary representations of dissociated groups. 

\begin{corollary}
    Every dissociated closed permutation group is of type I. 
\end{corollary}

\begin{proof}
    We use the terminology of \cite{BdlH}. Let $G\leq\Sym(\Omega)$ be a dissociated closed permutation group and let $\pi$ be a unitary representation of $G$. By Corollary \ref{cor.classification.dissociated.unirep}, there exists a set $I$, cardinals $(n_i)_{i\in I}$ and a family of mutually non-equivalent irreducible representations $(\pi_i)_{i\in I}$ of $G$ such that $\pi=\bigoplus_{i\in I} n_i\pi_i$. Then $\pi$ is type I by Remark 6.A.13 (2) in \cite{BdlH}. 
\end{proof}

\subsection{Property (T)}

A topological group $G$ has \textit{Property (T)} if there exists a compact subset $Q\subseteq G$ and $\varepsilon>0$ such that every unitary representation $G\to\UU(\HH)$ with a non-zero vector $\xi\in\HH$ satisfying 
\[\underset{g\in Q}{\sup}\lVert \pi(g)\xi-\xi\rVert \leq \varepsilon\lVert\xi\rVert\]
admits a non-zero invariant vector. The pair $(Q,\varepsilon)$ is called a Kazhdan pair. We describe here a useful criterion for proving Property (T) for closed permutation groups. It is due to Tsankov \cite{Tsankov} but is not explicitly formulated. We will not provide a proof, but a careful look at Section 6 (in particular Lemma 6.3, Proposition 6.4 and Lemma 6.5) of Tsankov’s paper allows one to extract the following result.

\begin{theorem}\label{thm.Property(T).Todor}
    Let $G\leq\Sym(\Omega)$ be a closed permutation group without algebraicity. Assume that every unitary representation of $G$ is a direct sum of irreducible ones and that every irreducible representation of $G$ is a subrepresentation of $G \curvearrowright \ell^2(\Omega^n)$ for some $n\leq 0$. Then $G$ has property $(T)$. 
\end{theorem}

Notice that a dissociated group satisfies the assumption of the above theorem. Indeed, if $G\leq\Sym(\Omega)$ is dissociated, then we know by Corollary \ref{cor.classification.dissociated.unirep} that every unitary representation of $G$ is a direct sum of irreducible ones, and that an irreducible one has the form $\Ind_{G_{\{A\}}}^G(\sigma)$ for some $A\subseteq\Omega$ and some irreducible unitary representation $\sigma \colon G_{\{A\}}\to\UU(\KK)$ which factors through the finite group $F\coloneqq G_{\{A\}}/G_A$. If $\lambda_F$ denotes the left-regular representation of $F$, it follows from the basic properties of induction (see e.g.\ \cite[Sec.~6]{Folland_1995}) that:
\begin{align*}
      \Ind_{G_{\{A\}}}^G(\sigma)\leq \Ind_{G_{\{A\}}}^G(\lambda_F) &\simeq \Ind_{G_{\{A\}}}^G(\Ind_{G_A}^{G_{\{A\}}}(1_{G_A}))  \\
      &\simeq \Ind_{G_A}^G(1_{G_A}) \\ &\simeq\ell^2(G/G_A)\\
      &\leq \ell^2(\Omega^n),
  \end{align*}
  where $n=\lvert A\rvert$. Therefore, Theorem \ref{thm.Property(T).Todor} applies and yields the following result. 

\begin{theorem}\label{thm.dissociated.property(T)}
    Every dissociated closed permutation group has Property (T).
\end{theorem}

A topological group has the strong Property (T) if there exists a Kazhdan pair $(Q,\varepsilon)$ with $Q$ finite. Evans and Tsankov proved in \cite{EvansTsankov} that every dissociated oligomorphic group has strong Property (T). 

\begin{question}
    Does every dissociated group have strong Property (T)?
\end{question}
To give a positive answer to this question, it would be sufficient to prove that every dissociated group $G\leq\Sym(\Omega)$ contains a free group of finite rank acting freely on $\Omega$, see e.g.\ \cite[Thm.~5.2]{BJJ24} or \cite{EvansTsankov}.

\subsection{The Howe-Moore property}

Another important rigidity phenomenon in harmonic analysis is the Howe-Moore property: the vanishing at infinity of the matrix coefficients of its unitary representations. While classically stated for locally compact groups, we will give a general definition for topological groups in Definition \ref{Def.HW}.

Historically, this property was proved for several families of locally compact groups. It was first observed by Howe, Moore and Zimmer in some Lie groups\footnote{Connected simple real Lie groups with finite center \cite{howe_asymptotic_1979}.} as well as in their algebraic counterparts\footnote{Isotropic simple algebraic groups over non-Archimedean local fields \cite{howe_asymptotic_1979}.}. A third class of groups with this property was identified by Lubotzky, Mozes and independently Pemantle in the context of groups acting on regular trees, a result later refined by Burger and Mozes\footnote{Topologically-simple closed subgroups of $\Aut(\mathbf{T})$ that act 2-transitively on the boundary of a bi-regular tree $\mathbf{T}$ where both arities are at least 3. \cite[Prop.~4.2]{burger_lattices_2000}.}. All these groups admit a so called Cartan decomposition, allowing for a unified proof that covers all known examples \cite{ciobotaru_unified_2015}. While the groups we consider in this paper are never locally compact, the Howe-Moore property can still be defined within the framework of Rosendal's large scale geometry \cite{Rosendal}. In this section, we show that some dissociated groups have this property. These new examples are not covered by the unified proof from the classical case.

Recall that an \textit{écart} on a topological group $G$ is a map $\delta\colon G^2 \to \R_+$ such that for every $x,y,z\in G$, we have $\delta(x,x) = 0$ and $\delta(x,z) \leqslant \delta(x,y) + \delta(y,z)$. It is \textit{left-invariant} if $\delta(xy,xz) = \delta(y,z)$ for every $x,y,z\in G$. A subset $B\subseteq G$ is \textit{coarsely bounded} (in the left coarse-structure) if its $\delta$-diameter is finite for every continuous left-invariant écart $\delta$ on $G$. We will write $\CC\BB(G)$ for the set of coarsely bounded subsets of $G$. The group $G$ is coarsely bounded if $G\in \CC\BB(G)$ and \textit{locally bounded} if $\CC\BB(G)$ contains a non-empty open set. Recall also that if $G$ is locally compact, $\CC\BB(G)$ coincides with the collection of relatively compact subsets of $G$.

A generalization of compactness that naturally arises in model theory is Roelcke-precompactness. A subset $B$ of a topological group $G$ is Roelcke-precompact if for every neighborhood $U$ of the identity in $G$, there exists a finite subset $F\subseteq G$ such that $B \subseteq UFU$. A group is \textit{Roelcke-precompact} if it is Roelcke-precompact as a subset of itself and \textit{locally Roelcke-precompact} if it admits a non-empty Roelcke-precompact open subset. The class of Roelcke-precompact groups contains all the automorphism groups of $\aleph_0$-categorical structures. 

Since Roelcke-precompact subsets are always coarsely bounded, Roelcke-precompact groups have a trivial large scale geometry. Locally Roelcke-precompact Polish groups, however, form a particularly nice class of groups for the study of their large scale geometry \cite[Chap.~3]{Rosendal}. In particular, they are characterized among Polish groups by the fact that their Roelcke completion is locally compact \cite{zielinski_locally_2021}. Moreover, a subset of a locally Roelcke-precompact Polish group is coarsely bounded if and only if it is Roelcke-precompact. Examples include the group $\Aut(\mathbf{T}_\infty)$ of automorphisms of the regular tree of countably-infinite valency or the isometry group of the integral Urysohn space $\Z\U$. See Section \ref{sec.examples} for more examples.

Now that this framework has been established, the Howe-Moore property naturally generalizes to general topological groups. To that aim, we will say that a function $f\colon G \to \mathbb{C}$ \textit{vanishes at infinity} if for every $\varepsilon>0$, the set $\{x\in G, |f(x)|\geqslant 0\}$ is coarsely bounded, and write $C_0(G)$ for the set of all such continuous functions. In what follows, given a unitary representation $G\to \UU(\HH)$, an \textit{invariant vector} is a fixed point of the action $G\curvearrowright \HH$.

\begin{definition}\label{Def.HW}
    A topological group $G$ has the \textit{Howe-Moore} property if for every unitary representation $G\to \UU(\HH)$ with no non-zero invariant vector, every matrix coefficient of $\pi$ vanishes at infinity, i.e.\ if for every $\xi, \eta \in \HH,$ the map $g\mapsto \langle \pi(g)\xi\mid \eta\rangle$ belongs to $C_0(G)$.
\end{definition}

We recover some properties that are well known when $G$ is a locally compact group.

\begin{lemma}\label{lem.properbounded}
    Let $G$ be a topological group with the Howe-Moore property. Then every proper open subgroup of $G$ is coarsely bounded.
\end{lemma}

\begin{proof}
Let $U\leqslant G$ be a proper open subgroup. Consider the associated quasi-regular representation $\lambda_{G/U} \colon G \to\UU(\ell^2(G/U))$. Note that this representation has non-zero invariant vectors (the constant functions) if and only if $U$ has finite index in $G$. We deal with the two cases separately.

Assume first that $U$ has infinite index in $G$. Write $\delta_U$ for the Dirac function at $U$ on $G/U$. Then $\delta_U \in \ell^2(G/U)$ and, by the Howe-Moore property, the set
$$B \coloneqq \{g\in G\ \left|\langle\lambda_{G/U}(g)\delta_U\mid\delta_U\rangle\right| \geqslant 1 \}$$
is coarsely bounded. Since for every $g \in U$, $\langle \lambda_{G/U}(g)\delta_U\mid\delta_U\rangle = 1$, we have $U \subseteq B$ hence $U$ is coarsely bounded.

Assume now that $U$ has finite index in $G$ and restrict $\lambda_{G/U}$ to the orthogonal complement $\ell^2_0(G/U)$ of the constant functions. By construction, this representation has no non-zero invariant vector. Denoting again by $\delta_U$ the Dirac function at $U$ on $G/U$, define $\xi\coloneqq \delta_U - 1/[G:U]$. Then $\xi\in \ell^2_0(G/U)$ and by the Howe-Moore property, the set
$$B \coloneqq \{g\in G, |\langle \lambda_{G/U}(g)\xi\mid\xi\rangle|\geqslant 1 \}$$
is coarsely bounded. Since $U$ is proper, we have for every $g \in U$, $\langle \lambda_{G/U}(g)\xi\mid\xi\rangle = 2 - 2/[G:U] \geq 1$. Thus $U \subseteq B$ hence $U$ is coarsely bounded. 
\end{proof}

Since translates and finite unions of coarsely bounded sets are coarsely bounded, we get the following consequence.

\begin{corollary}
        Let $G$ be a non-bounded topological group which has the Howe-Moore property. Then every proper open subgroup of $G$ has infinite index.
\end{corollary}

The main result of this section is the following. 

\begin{theorem}\label{thm.Howe.Moore.dissociated}
Let $G \leqslant \Sym(\Omega)$ be a closed permutation group which acts without fixed point on $\Omega$. If $G$ is dissociated, then the following properties are equivalent. 

    \begin{enumerate}[label=(\roman*)]
        \item\label{item.1} $G$ has the Howe-Moore property.
        \item\label{item.2} Every proper open subgroup of $G$ is coarsely bounded.
        \item\label{item.3} For every non-empty finite subset $A \subseteq \Omega$, $\Nm G A$ is coarsely bounded.
        \item\label{item.4} For every non-empty finite subset $A \subseteq \Omega$, $G_A$ is coarsely bounded.
        \item\label{item.5} For every $a\in \Omega$, $G_a$ is coarsely bounded. 
    \end{enumerate}
\end{theorem}

\begin{proof}
The implication \ref{item.1} $\Rightarrow$ \ref{item.2} is proved in Lemma \ref{lem.properbounded}. The equivalences \ref{item.2} $\Leftrightarrow$ \ref{item.3} $\Leftrightarrow$ \ref{item.4} $\Leftrightarrow$ \ref{item.5} are straightforward by recalling that for a dissociated group $G$ acting without fixed point on $\Omega$, every open subgroup of $G$ lies between the pointwise stabilizer and the setwise stabilizer of a unique finite subset of $\Omega$ (Lemma \ref{lem.caracterisation.noAlg.WEI} and \ref{lem.dissociated.noAlg.WEI}). 

Let us finally prove that \ref{item.3} $\Rightarrow$ \ref{item.1}. Notice that a direct sum of representations whose matrix coefficients are all in $C_0(G)$ have the same property. By the classification of unitary representations for dissociated groups (Corollary \ref{cor.classification.dissociated.unirep}), it suffices to prove that the matrix coefficients of every irreducible unitary representation of $G$ without non-zero invariant vector belong to $C_0(G)$.

The only irreducible representation of a topological group that has non-zero invariant vectors is the trivial representation on $\mathbb{C}$. So let $\pi$ be a nontrivial irreducible representation of $G$ on a Hilbert space $\HH$. By Corollary \ref{cor.classification.dissociated.unirep}, there exists $A\subseteq\Omega$ a finite subset and $\sigma \colon G_{\{A\}}\to\UU(\KK)$ an irreducible unitary representation which is trivial on $G_A$ and such that $\pi \simeq \Ind_{G_{\{A\}}}^G(\sigma)$. Note that $A$ is non-empty since $\pi$ is nontrivial. Fix a transversal $(h_i)_{i\in I}$ for $G/G_{\{A\}}$, so that for all $\xi,\eta\in\HH$, the scalar product $\langle \xi\mid \eta\rangle$ is given by:
\[\langle \xi \mid \eta \rangle = \sum_{i\in I}\langle \xi(h_i)\mid \eta(h_i)\rangle.\]
Define
\[D\coloneqq\{\xi\in\HH\colon \supp(\xi)\text{ is contained in a finite union of left cosets of }G_{\{A\}}\}.\]
Then $D$ is a dense subset of $\HH$. Fix a non-zero vector $\xi\in D$ and let $B\coloneqq\supp(\xi)$. Notice that $B$ is coarsely bounded by assumption (recall that $A$ is non-empty). Then for all $g\in G$, we have 
\begin{align*}
    \langle \pi(g)\xi\mid \xi\rangle &= \sum_{i\in I}\langle \xi(g\inv h_i)\mid \xi(h_i)\rangle \\
    & = \sum_{\substack{i\in I \\ h_i\in B\cap gB}}\langle \xi(g\inv h_i)\mid \xi(h_i)\rangle 
\end{align*}
Therefore, for every $g\notin BB\inv$, $B\cap gB = \emptyset$ and $\langle \pi(g)\xi\mid \xi\rangle =0$. Since $BB\inv$ is coarsely bounded, we obtain that for every $\xi\in D$, the map $g\mapsto\langle\pi(g)\xi\mid\xi\rangle$ belongs to $C_0(G)$. Since $D$ is dense in $\HH$, this shows that for all $\xi,\eta\in\HH$, the matrix coefficient $g\mapsto\langle \pi(g)\xi\mid\eta\rangle$ belongs to $C_0(G)$. That is, $G$ has the Howe-Moore property. 
\end{proof}

Given a countable, additive subsemigroup of $\R_+$, one can consider the Fraïssé limit of finite metric spaces $(X,d)$ such that $d(X\times X)\subseteq\Delta\cup\{0\}$. This is called the Urysohn $\Delta$-metric space and is denoted by $\U_\Delta$. Isometry groups of Urysohn $\Delta$-metric spaces provide new examples of Polish non-locally compact groups satisfying the Howe-Moore property. 

\begin{corollary}
    Let $\Delta$ be a countable, additive subsemigroup of $\R_+$. Then $\Isom(\U_\Delta)$ has the Howe-Moore property.
\end{corollary}

\begin{proof}
   Notice that countable, additive subsemigroup of $\R_+$ are unbounded, so $\Isom(\U_\Delta)$ is locally bounded but not coarsely bounded by \cite[Lem.~4.3]{BJJ24}. We proved in \cite[Thm.~4.12]{BJJ24} that $\Isom(\U_\Delta)$ is dissociated. Moreover, pointwise stabilizers $\Isom(\U_\Delta)_A$ of non-empty finite subsets $A\subseteq\U_\Delta$ are coarsely bounded by Theorem 6.31 and Example 6.32 of \cite{Rosendal}. The conclusion thus follows from Theorem \ref{thm.Howe.Moore.dissociated}. 
\end{proof}

Just as in the locally compact case, the Howe-Moore property has a direct translation in ergodic theory. A Boolean p.m.p.\ action $\alpha \colon G\to\Aut(X,\mu)$ of a topological group $G$ is \emph{mixing} if for every measurable sets $Y,Z\subseteq X$, the map \[g\mapsto \mu(Y\cap\alpha(g)Z)-\mu(Y)\mu(Z)\] belongs to $C_0(G)$. Moreover, a Boolean p.m.p.\ action $\alpha \colon G\to\Aut(X,\mu)$ is:
\begin{itemize}
    \item ergodic if and only if its Koopman representation $\kappa_0\colon G\to\UU(\L^2(X,\mu)\ominus\C)$ has no nontrivial invariant vector,
    \item mixing if and only if for all $\xi,\eta\in\L^2(X,\mu)\ominus\C$, the matrix coefficient $g\mapsto \langle \kappa_0(g)\xi\mid\eta\rangle$ of $\kappa_0$ belongs to $C_0(G)$.
\end{itemize} 
Thus, a topological group $G$ has the Howe-Moore property if and only if every ergodic Boolean p.m.p.\ action of $G$ is mixing.

\section{Ergodic theory of dissociated groups}

\subsection{Non-singular actions}

Let $(X,\nu)$ be a $\sigma$-finite nontrivial measure space. We denote by $\Aut(X,[\nu])$ the group of bimeasurable bijections which preserve the class of $\nu$, where two such bijections are identified if they coincide on a conull set. Since $\nu$ is equivalent to a probability measure, we may always assume that $\nu$ is a probability measure. Recall that there is a natural embedding $\kappa\colon\Aut(X,[\nu])\to\UU(\L^2(X,\nu))$ defined as follows: for all $g\in\Aut(X,[\nu])$ and $f\in\L^2(X,\nu)$, 
\begin{align}\label{eq.koopman}
    \kappa(g)f \colon x\mapsto \sqrt{\frac{dg_*\nu}{d\nu}(x)}f(g\inv x)
\end{align}
where for every $g\in G$, $g_*\nu$ is the pushforward of $\nu$ by $g$ and $dg_*\nu/d\nu$ is the Radon-Nikodym derivative of $g_*\nu$ with respect to $\nu$. The image of $\kappa$ is closed and we equip $\Aut(X,[\nu])$ with the topology induced from the strong operator topology on $\UU(\L^2(X,\mu))$. This is a Polish group. For more details on $\Aut(X,[\nu])$ as a topological group, we refer to Section 4 of \cite{danilenko}. 

A \textit{non-singular} action of a topological group $G$ on a probability space $(X,\nu)$ is a continuous homomorphism $\alpha\colon G\to\Aut(X,[\nu])$. It is \textit{ergodic} if every measurable set $Y\subseteq X$ such that for all $g\in G,\  \nu(\alpha(g)Y\triangle Y)=0$ is either null of conull. A non-singular ergodic action $G\to\Aut(X,[\nu])$ is of \textit{type I} if $\nu$ is purely atomic, of \textit{type II} if there exists a diffuse, $\sigma$-finite, $G$-invariant measure $\mu$ which is equivalent to $\nu$ and of \textit{type III} otherwise. 

\begin{theorem}
    Let $G\leq\Sym(\Omega)$ be a closed permutation group. If $G$ is dissociated, then $G$ admits no ergodic non-singular action of type III. 
\end{theorem}

In fact, a stronger result holds. In order to state it properly, we need the definition of induction in the setting of p.m.p.\ actions (there is actually a more general definition of induction in the setting of non-singular actions, however it is not relevant for our purpose). Let $G$ be a Polish group and $H\leq G$ be an open subgroup. We endow the countable space $G/H$ with the counting measure. Let $\alpha\colon H\to\Aut(X,\mu)$ be a Boolean p.m.p.\ action of $H$ on a standard probability space $(X,\mu)$. Let $\MAlg(X,\mu)$ be the measure algebra of $(X,\mu)$. Define  \[\mathfrak{A}\coloneqq\{f\colon G\to\MAlg(X,\mu)\colon \forall g\in G, \forall h\in H, f(gh)=\alpha(h\inv)f(g)\}.\]
Define two binary operations $\triangle$ and $\cap$ on $\mathfrak{A}$ as follows: for all $f_1,f_2\in\mathfrak{A}$
\[f_1\triangle f_2\colon g\mapsto f_1(g)\triangle f_2(g)\]
\[f_1\cap f_2\colon g\mapsto f_1(g)\cap f_2(g).\]
Let $0_\mathfrak{A}$ and $1_\mathfrak{A}$ be the functions which are identically equal to $\emptyset$ and $X$ respectively. Notice that for every $f\in\mathfrak{A}$, the map $g\mapsto \mu(f(g))$ is constant on each left $H$-coset. Denote by $\mu(f(q))$ its value on the coset $q\in G/H$. Define the following map $\tilde\mu \colon \mathfrak{A}\to[0,+\infty]$ by 
\[\tilde{\mu}(f)\coloneqq \sum_{q\in G/H}\mu(f(q)).\]

It is left to the cautious reader to check that $(\mathfrak{A},\triangle,\cap,0_\mathfrak{A},1_\mathfrak{A},\tilde{\mu})$ is a measure algebra, which is separable in its measure-algebra topology. Therefore, there exists a standard probability space $(Y,\nu)$ such that $\mathfrak{A}=\MAlg(Y,\nu)$. The continuous homomorphism $G\to\Aut(\mathfrak{A})$ given by the $G$-action by precomposition on $\mathfrak{A}$ thus provides a Boolean p.m.p.\ action $G\to\Aut(Y,\nu)$ which is denoted by $\Ind_{H}^G(\alpha)$ and is called the \textit{Boolean p.m.p.\ $G$-action induced by} $\alpha$. 

The proof of Tsankov's classification of non-singular actions of Roelcke-precompact non-Archimedean Polish groups \cite[Thm.~3.4]{tsankov2024nonsingularprobabilitymeasurepreservingactions} adapts without effort to the context of dissociated groups. Due to the substantial similarities, we omit the proof of the following result.

\begin{theorem}\label{thm.nonsingular.induced}
    Let $G\leq\Sym(\Omega)$ be a closed permutation group. Assume that $G$ is dissociated. Let $(X,\nu)$ be a standard probability space and $\alpha \colon G\to\Aut(X,[\nu])$ be a non-singular action. Then $\alpha$ is isomorphic to a countable union $\bigsqcup_{i\in I}\Ind_{G_{\{A_i\}}}^G(\alpha_i)$, where for every $i\in I$, $A_i$ is a finite subset of $\Omega$ and $\alpha_i$ is a Boolean p.m.p.\ action of $G_{\{A_i\}}$.  \end{theorem}

\subsection{Stabilizer rigidity}

This section goes over the main result of \cite{jahel2023stabilizers}, where are studied Borel p.m.p.\ actions of closed permutation groups through the perspective of the stabilizers associated to these actions. Here, by 
a \textit{Borel p.m.p.\ action} of a Polish group $G$ we mean a Borel action $G\times X\to X$ on a standard Borel space $X$ together with a Borel $G$-invariant probability measure $\mu$. These are sometimes called spatial p.m.p.\ actions to emphasize the difference between Boolean p.m.p.\ actions. To differentiate them, Borel p.m.p.\ actions will be denoted by $G\curvearrowright(X,\mu)$ whereas Boolean p.m.p.\ actions by $G\to\Aut(X,\mu)$. A Borel p.m.p.\ action $G\curvearrowright(X,\mu)$ always yields a Boolean p.m.p.\ action $G\to \Aut(X,\mu)$ (continuity is automatic) and we will say that the Borel p.m.p.\ action $G\curvearrowright(X,\mu)$ is dissociated if the associated Boolean p.m.p.\ action is. A Borel p.m.p.\ action $G\curvearrowright(X,\mu)$ is:
\begin{itemize}
    \item \emph{essentially transitive} if there exists a $G$-orbit $O\subseteq X$ such that $\mu(O)=1$ (recall that orbits are Borel \cite[Thm.~15.14]{KechrisCDST1995}).
    \item \emph{essentially free} if the set $\{x\in X\colon \forall g\in G\setminus\{1_G\}, g\cdot x\neq x\}$ is conull. 
\end{itemize}

A subgroup $G\leq \Sym(\Omega)$ is primitive if it acts transitively on $\Omega$ and there are no $G$-invariant equivalence relation on $\Omega$ apart from equality and $\Omega\times\Omega$. One can easily check that if $G$ has no algebraicity and weakly eliminates imaginaries, then $G$ is primitive (Corollary 3.6 in \cite{jahel2023stabilizers}). We can now state the main result of \cite{jahel2023stabilizers}, which is a permutation group variant on Stuck-Zimmer's Theorem for locally compact groups. In what follows, given a Borel p.m.p.\ action $G\curvearrowright (X,\mu)$ and $x \in X$, we will denote by $\Stab(x)\coloneqq\{g \in G, \ g(x)=x\}$ the stabilizer of $x$ in $G$.

\begin{theorem}[Theorem 1.4 of \cite{jahel2023stabilizers}]\label{thm.pmprigiddissociatedexpanded}
    Let $G<\Sym(\Omega)$ be a proper, closed permutation group. If $G$ has no algebraicity and is primitive, then for every dissociated Borel p.m.p.\ ergodic action $G\curvearrowright(X,\mu)$, the following hold:
    \begin{itemize}
        \item either $\Stab(x)\curvearrowright\Omega$ has a fixed point for $\mu$-almost every $x\in X$ and in this case $G\curvearrowright(X,\mu)$ is essentially free,
        \item or $\Stab(x)\curvearrowright\Omega$ has no fixed point for $\mu$-almost every $x\in X$ and in this case $G\curvearrowright(X,\mu)$ is essentially transitive. 
    \end{itemize}
\end{theorem}

For dissociated groups, we therefore have the following striking dichotomy, reminiscent of Stuck-Zimmer's Theorem for higher rank simple Lie group \cite{StuckZimmer}. 

\begin{theorem}
    Let $G$ be a proper, transitive, closed permutation group. If $G$ is dissociated, then every Borel p.m.p.\ ergodic action of $G$ is either essentially transitive or essentially free.
\end{theorem}

Notice that these results concern Borel p.m.p.\ actions. Indeed, the notions of essential freeness and essential transitivity make no sense for Boolean p.m.p.\ action. In fact, there even exists two Borel p.m.p.\ actions of $\Sym(\Omega)$, one being essentially free, the other one being essentially transitive, whose associated Boolean p.m.p.\ actions are Booleanly isomorphic \cite[Sec.~4.3]{HoareauLeMaître}.

\subsection{Exchangeable processes}

Let $G$ be a closed permutation group on $\Omega$ and let $(Z,\zeta)$ be a standard probability space. The \emph{Bernoulli shift} over $G$ with base space $(Z,\zeta)$ is the Borel p.m.p.\ action of $G$ on $(Z,\zeta)^\Omega$ defined for every $g\in G$ and $(z_\omega)_{\omega\in\Omega}\in Z^\Omega$ by
\[g\cdot (z_\omega)_{\omega\in\Omega}=(z_{g\inv(\omega)})_{\omega\in\Omega}.\]

For every subset  $A\subseteq\Omega$, let $\mathrm{proj}_A\colon Z^\Omega\to Z^A$ be the restriction map and let $\sigma(\mathrm{proj}_A)$ be the $\sigma$-algebra generated by $\mathrm{proj}_A$.

\begin{lemma}\label{lem.claim}
    Let $G\leq\Sym(\Omega)$ be a closed permutation group without algebraicity and let $G\curvearrowright(Z,\zeta)^\Omega$ be the Bernoulli shift. For every finite subset $A\subseteq\Omega$, the $\sigma$-algebra $\FF_A$ of $G_A$-invariant subsets coincides (up to null sets) with $\sigma(\mathrm{proj}_A)$. Consequently, $G\curvearrowright(Z,\zeta)^\Omega$ is dissociated.
\end{lemma}

\begin{proof}
Fix $A\subseteq\Omega$ finite. For every Borel subset $Y$ of $Z^A$, it is clear that $\mathrm{proj}_A\inv(Y)$ is $G_A$-invariant. So the $\sigma$-algebra generated by $\mathrm{proj}_A$ is contained in $\FF_A$. For the converse, fix $Y\in\FF_A$. Fix $H$ a countable dense subgroup of $G_A$ and let $Y_0\subseteq Y$ be a subset of full measure such that for every $h\in H$, the set $h\cdot Y_0$ is equal to $Y_0$ (here this is a true equality, not only up to null sets). For every $x\in Z^A$, define
\[Y_0^x\coloneqq\{\mathrm{proj}_{A^c}((z_n)_{n\geq 0})\colon (z_n)_{n\geq 0}\in Y_0\text{ and }(z_a)_{a\in A}=x\}.\]
Then $Y_0^x$ is invariant under $H$ for every $x \in Z^A$. Since $H$ is dense in $G_A$, they have the same orbits on $\Omega\setminus A$. Since $G$ has no algebraicity, all of these orbits are infinite. This implies that $H$ acts ergodically on $(Z,\zeta)^{\Omega\setminus A}$ (see e.g.\ \cite[Prop.~2.1]{KechrisTsankov} for a proof). Thus, $\zeta^{\Omega\setminus A}(Y_0^x)\in\{0,1\}$ for every $z\in Z$. If $E$ denotes the set of $x\in Z^A$ such that $\zeta^{\Omega\setminus A}(Y_0^x)=1$, then $Y_0\triangle\mathrm{proj}_A\inv(E)$ has measure $0$. This shows the second inclusion.

To prove dissociation of the Bernoulli shift, take $A,B\subseteq\Omega$ finite. By the previous argument, $\FF_A$, $\FF_B$ and $\FF_{A\cap B}$ coincide (up to null sets) respectively with the $\sigma$-algebras generated by $\mathrm{proj}_A$, $\mathrm{proj}_B$ and $\mathrm{proj}_{A\cap B}$. But it is clear that $\sigma(\mathrm{proj}_A)$ and $\sigma(\mathrm{proj}_B)$ and independent conditionally on $\sigma(\mathrm{proj}_{A\cap B})$, which shows that the Bernoulli shift $G\curvearrowright(Z,\zeta)^\Omega$ is dissociated. 
\end{proof}

\begin{corollary}\label{cor.deFinetti}
    Let $G\leq\Sym(\Omega)$ be a transitive, closed permutation group without algebraicity and let $Z$ be a standard Borel space. Let $\mu$ be a $G$-invariant, Borel, ergodic probability measure on $Z^\Omega$. Then $\mu$ is dissociated if and only if $\mu=\lambda^{\otimes\Omega}$ for some Borel probability measure $\lambda$ on $Z$.
\end{corollary}
\begin{proof}
Suppose $\mu$ is a $G$-invariant ergodic Borel probability measure on $Z^\Omega$, which is dissociated. By $G$-invariance of $\mu$ and transitivity of $G\curvearrowright \Omega$, the one-dimensional marginals of $\mu$ are all equal. Let $\zeta$ denote one of them. Then for every measurable $A_1,\dots,A_n\subseteq Z$, dissociation yields that $\mu(A_1\times\dots\times A_n\times Z\times ...)=\zeta(A_1)\dots\zeta(A_n)$. Therefore $\mu$ is a product measure. The converse is given by Lemma \ref{lem.claim}.
\end{proof}

Consequently, de Finetti's theorem \cite{deFinetti} can be considered as the first historical instance of dissociation, for product actions of the group $\Sym(\Omega)$. Indeed, by Corollary \ref{cor.deFinetti}, de Finetti's theorem is equivalent to saying that for every $\Sym(\Omega)$-invariant ergodic probability measure on $Z^\Omega$, the Borel p.m.p.\ action $\Sym(\Omega)\curvearrowright (Z^\Omega,\mu)$ is dissociated. A similar result has been obtained by Ryll-Nardzewski for the group $\Aut(\Q,<)$ \cite{Ryll-Nardzewski}. It was later generalized to a large class of closed permutation groups, see \cite{JT} and \cite{tsankov2024nonsingularprobabilitymeasurepreservingactions}. 

\section{How to obtain dissociation}\label{sec.uniform.non.algebraicity}

\subsection{Via approximating sequences}\label{sec.approx}

In this section, we recall a method to obtain dissociation that we introduced in \cite{BJJ24} for studying the automorphism group of the rational Urysohn space. 

Let $\Omega_0\subseteq\Omega_1$ be two countably infinite sets and let $G_0, G_1$ be closed permutation groups over $\Omega_0$ and $\Omega_1$ respectively. An \emph{extension embedding} of $G_0$ into $G_1$ is an embedding of topological groups $\theta \colon G_0\hookrightarrow G_1$ such that for every $g\in G_0$ and $x\in\Omega_0$, we have $\theta(g)(x)=g(x)$. 

Given $G\leqslant \Sym(\Omega)$ a closed permutation group, an \textit{approximating sequence} for $G$ is the data of an increasing sequence $\Omega_0\subseteq\Omega_1\subseteq\dots\subseteq\Omega$ of infinite subsets with $\Omega=\bigcup_{n\geq 0}\Omega_n$, a sequence of closed permutation groups $G_n\leq\Sym(\Omega_n)$ and a sequence of extension embeddings $\theta_n \colon G_n\hookrightarrow G_{n+1}$ such that $\bigcup_{n\geq 0}\iota_n(G_n)$ is a dense subgroup of $G$. Here, $\iota_n \colon G_n\hookrightarrow\Sym(\Omega)$ denotes the natural extension embedding obtained by composing the extension embeddings $\theta_n$.

\begin{theorem}[{\cite[Thm.~3.12]{BJJ24}}]\label{thm.approx.sequence} 
Let $G\leq\Sym(\Omega)$ be a closed permutation group with an approximating sequence $G_0\hookrightarrow G_1\hookrightarrow\dots\hookrightarrow G$. Assume that for every $n\geq 0$, the closed permutation group $G_n\leq\Sym(\Omega_n)$ is dissociated. Then the closed permutation group $G\leq\Sym(\Omega)$ is dissociated. 
\end{theorem}

\begin{remark}
    Extension embeddings are a stronger version of what are called extensive embedding in \cite{KSW}. Indeed, we do ask for more than simply extensivity, as we require our embeddings to be topological. Yet another related notion appears in \cite{BartosKubis}, where they discuss extensibility of structures, however the notion of extension embedding they use for their result is even less restrictive than the one in \cite{KSW}. The topological version we use is necessary for the dynamical results we obtain.
\end{remark}

\subsection{Via an amalgamation property}

The methods presented in this section are heavily inspired from Section 4 in \cite{tsankov2024nonsingularprobabilitymeasurepreservingactions}.

\subsubsection{Tail subspaces and dissociation}

In this section, we introduce for every finite subset $A\subseteq\Omega$ a \textit{tail subspace} $\TT_A$ associated with a unitary representation of a closed permutation group $G\leq\Sym(\Omega)$. Our definition is inspired by the classical definition of the tail $\sigma$-algebra for an exchangeable random process, see Remark \ref{rem.tail.sigma.algebra}. The key idea in this section is to obtain dissociation by proving that the tail space $\TT_A$ coincides with the space of $G_A$-invariant vectors. 

\begin{definition}
    Let $G\leq\Sym(\Omega)$ be a closed permutation group and $A \subseteq \Omega$ a finite subset. The \textit{tail subspace} $\TT_A$ associated with a unitary representation $\pi \colon G\to\UU(\HH)$ is defined by 
    \[\TT_A\coloneqq \bigcap_{\substack{C\subseteq\Omega \text{ cofinite}\\ A \subseteq C}}\HH_C.\]
\end{definition}

Note that by definition, we always have $\HH_A\subseteq\TT_{A}$. In the sequel, we will need a finer notion of conditional orthogonality than the one we present in the introduction. Let $\HH$ be a Hilbert space. For any closed subspace $\GG\subseteq \HH$, we denote by $p_\GG$ the orthogonal projection onto $\GG$. 

\begin{definition}
    Let $\GG,\KK,\LL$ be three closed subspaces of $\HH$. We say that $\GG$ and $\LL$ are orthogonal conditionally on $\KK$ if one of the following equivalent assertions is satisfied:
\begin{enumerate}[label=(\roman*)]
   \item $p_{\KK^\bot}\GG\bot p_{\KK^\bot}\LL$.
    \item $p_{\GG}p_{\KK}p_{\LL}=p_{\GG}p_{\LL}$. 
    \item $p_{\LL}p_{\KK}p_{\GG}=p_{\LL}p_{\GG}$
 \end{enumerate}
If this holds, we write $\GG\bot_\KK\LL$. 
\end{definition}
 Notice that if $\KK\subseteq\GG\cap\LL$, then we recover the definition given in the introduction: $\GG\bot_{\KK}\LL$ is equivalent to $p_{\GG}p_{\LL}=p_{\LL}p_{\GG}=p_{\KK}$. In this case, we readily have that $\KK=\GG\cap\LL$.

The following theorem provides a weak form of dissociation.

\begin{theorem}
    Let $G\leq \Sym(\Omega)$ be a closed subgroup without algebraicity. Let $\pi \colon G\to\UU(\HH)$ be a unitary representation. Then $\HH_A\bot_{\TT_{A\cap B}}\HH_B$ for all finite subsets $A,B\subseteq\Omega$. 
\end{theorem}

\begin{proof}
Fix $A,B\subseteq\Omega$ finite and let $\GG\coloneqq\HH_A$, $\KK=\TT_{A\cap B}$ and $\LL=\HH_B$. Let us prove that $p_{\GG}p_{\KK}p_{\LL}=p_{\GG}p_{\LL}$. Define $\KK'=\overline{\text{Vect}(\GG,\KK)}$. Notice that if $p_{\KK'}p_{\LL}=p_{\KK}p_{\LL}$, then $\GG\bot_{\KK}\LL$. Indeed 
\begin{align*}
 (p_{\KK'}p_{\LL}=p_{\KK}p_{\LL}) &\Rightarrow (p_\GG p_\KK p_\LL=p_\GG p_{\KK'}p_\LL =p_\GG p_{\LL} ) \\
    & \Rightarrow \GG\bot_\KK\LL.
\end{align*} 
where the rightmost equality holds because $\GG\subseteq\KK'$. So let us prove that for all $\xi\in\LL$,
\begin{align}\label{eqn.independenceovertail}
         p_{\KK'}\xi =p_{\KK}\xi. 
     \end{align} 
     Since $\KK$ is a subspace of $\KK'$, we have $\lVert p_{\KK}\xi\rVert\leq \lVert p_{\KK'}\xi\rVert$ and \eqref{eqn.independenceovertail} is equivalent to $\lVert p_{\KK'}\xi\rVert=\lVert p_{\KK}\xi\rVert$. We thus prove the reverse inequality.
     
     Fix a strictly decreasing sequence $C_0\supseteq C_1\supseteq \dots$ of cofinite subsets of $\Omega$ whose intersection $\bigcap_{n\geq 0}C_n$ is $A\cap B$. For every $n\geq 0$, let $\KK_n \coloneqq \HH_{C_n}$. Then $(\KK_n)_{n\geq 0}$ forms a decreasing sequence of closed subspaces. Moreover, the intersection $\bigcap_{n\geq 0}\KK_n$ is equal to $\KK$ and $p_{\KK_{n}}\to p_{\KK}$ in the strong operator topology. Since $G$ has no algebraicity, Neumann's Lemma \cite[Thm.~1]{Birch_1976} gives for every $n\geq 0$ some element $g_n\in G_B$ such that $g_n(A)$ is contained in $C_n$. In particular, $\pi(g_n)\GG=\pi(g_n)\HH_A=\HH_{g_n(A)}$ is contained in $\KK_{n}$. 
    
    For all $n\geq 0$, define $\KK_n'\coloneqq \pi(g_n)\KK'$. Notice that \begin{align*}
        \KK_n'&=\overline{\text{Vect}(\pi(g_n)\GG,\pi(g_n)\KK)}\\
            &=\overline{\text{Vect}(\HH_{g_n(A)},\KK)},
    \end{align*}
    where the last equality holds since $\KK$ is a $G_{A\cap B}$-invariant subspace and $g_n\in G_B\leq G_{A\cap B}$. Therefore, $\KK_n'$ is a subspace of $\KK_{n}$. For all $\xi \in\LL$ and $n\geq 0$, we have
     \[
         \lVert p_{\KK'}\xi\rVert  = \lVert \pi(g_n)p_{\KK'}\xi\rVert = \lVert  p_{\KK_n'}\pi(g_n)\xi\rVert  =\lVert p_{\KK_n'}\xi\rVert
         \leq \lVert p_{\KK_{n}}\xi \rVert \to \lVert p_{\KK}\xi\rVert,
     \]
     which concludes the proof.
\end{proof}

Therefore, if $\TT_A=\HH_A$ for every finite $A\subseteq\Omega$, then the unitary representation is dissociated. We can actually prove the converse of this statement, which will allow us to get a nice equivalence in Corollary \ref{cor.tail.dissociation.equivalence}. 

\begin{lemma}
    Let $G\leq\Sym(\Omega)$ be a closed permutation group. Let $\pi \colon G\to\UU(\HH)$ be a dissociated unitary representation. Then for every $A\subseteq\Omega$ finite, we have $\TT_A=\HH_A$. 
\end{lemma}

\begin{proof}
Let $A\subseteq\Omega$ be a finite subset. Let $(C_n)$ be a decreasing sequence of cofinite subsets of $\Omega$ such that $\bigcap_{n\in \N} C_n = A$. Then $\bigcap_{n\in \N} \HH_{C_n} = \TT_A$ and $p_{C_n} \to p_{\TT_A}$ in the strong operator topology.

Define $B_n \coloneqq A \cup \left(\Omega\setminus C_n\right)$ for every $n\in \N$. Then $(B_n)$ is an increasing sequence of finite sets whose union is $\Omega$. By continuity of $\pi$, we get that $p_{B_n} \to \operatorname{id}$ in the strong operator topology. Moreover, $B_n \cap C_n = A$ for every $n \in \N$. Thus, using dissociation, we have
$p_{C_n}p_{B_n} = p_A$ for every $n\in \N$. Note that $p_{C_n}p_{B_n} \to p_{\TT_A}$ in the strong operator topology (these operators being projections, they are uniformly bounded). At the limit, we get $p_{\TT_A} = p_A$ and $\TT_A = \HH_A$.
\end{proof}

\begin{corollary}\label{cor.tail.dissociation.equivalence}
    Let $G\leq\Sym(\Omega)$ be a closed permutation group without algebraicity. Let $\pi \colon G\to\UU(\HH)$ be a unitary representation. Then $\pi$ is dissociated if and only if for every finite subset $A\subseteq\Omega$, we have $\HH_A=\TT_A$.
\end{corollary}

\begin{remark}\label{rem.tail.sigma.algebra}
    Every result obtained in this section has an analogue in the context of Boolean p.m.p.\ actions. Let $G$ be a closed subgroup of $\Sym(\Omega)$ and $G\to\Aut(X,\mu)$ be a Boolean p.m.p.\ action. For every finite subset $A\subseteq\Omega$, define the following $\sigma$-algebra
    
    \[\TT_A\coloneqq\bigcap_{C\subseteq\Omega\text{ cofinite}}\sigma(\FF_{A\cup B}\colon B\subseteq C\text{ finite}).\] 
    When $A$ is the empty set, we call $\TT_\emptyset$ the tail $\sigma$-algebra of the action. One  checks that the proofs of this section can be adapted in a straightforward way to get analogous results in that context, leading to the following result: \emph{if $G$ has no algebraicity, then the Boolean p.m.p.\ action $\alpha$ is dissociated if and only if for every finite subset $A\subseteq\Omega$, we have $\FF_A=\TT_A$}.
    
    In particular, for a Boolean p.m.p.\ ergodic, dissociated action, the tail $\sigma$-algebra is trivial. Notice also that for $\Sym(\Omega)$ acting on a product probability space $(Z,\zeta)^\Omega$, triviality of the tail $\sigma$-algebra and of the $\sigma$-algebra of invariant subsets (a.k.a.\ exchangeable subsets in that context) correspond respectively to Kolmogorov and Hewitt-Savage $0-1$ laws.
\end{remark}

\subsubsection{Strong cofinite amalgamation over countable subsets}

Recall that, in the context of Fraïssé limits, the non-algebraicity condition we have been working with throughout this paper is equivalent to the Fraïssé class having the \textit{strong amalgamation property} (see Section 2.7 of \cite{Cameron}). We now introduce a reinforcement of this property by allowing infinite amalgamation bases and show it implies dissociation of the limit.

Let $\M$ be a countable relation structure and let $\sigma\mathrm{Age}(\M)$ be the set of isomorphism classes of countable (that is, finite or infinite) substructures of $\M$.

\begin{definition}\label{defn:SAPAS}
    We say that $\M$ has the \emph{strong cofinite amalgamation property over countable subsets} (abbrev.\ \emph{$\sigma$-SAP}) if, whenever $\A,\B_1,\B_2\in\sigma\mathrm{Age}(\M)$ and $f_i \colon \A\rightarrow \B_i$ are embeddings with $\lvert \dom(\B_i)\setminus \dom(f_i(\A))\rvert <+\infty$, there exist $\mathbf{C}\in\sigma\mathrm{Age}(\M)$ and embeddings $g_i \colon \B_i\rightarrow \mathbf{C}$ so that the following hold:
    \begin{enumerate}
        \item the diagram  \begin{tikzcd}
                      & \mathbf{C}                                      &                        \\
\B_1 \arrow[ru, "g_1"] &                                        & \B_2 \arrow[lu, "g_2"'] \\
                      & \A \arrow[lu, "f_1"] \arrow[ru, "f_2"'] &                       
\end{tikzcd} commutes,

\item for every $b_1\in \dom(\B_1)$ and $b_2\in\dom(\B_2)$ satisfying $g_1(b_1)=g_2(b_2)$, there exists $a\in\dom(\A)$ satisfying $f_1(a)=b_1$ and $f_2(a)=b_2$. 
    \end{enumerate}
\end{definition}

\begin{remark}\label{rem.omega.SAP.examples}
    For \emph{integral} metric spaces, there is a canonical amalgamation over a common non-empty subspace (of any cardinality!). Therefore, several natural Fraïssé limits built out of classes of finite integral metric spaces satisfy $\sigma$-SAP. This is for instance the case for the integral Urysohn space $\Z\U$, but other examples will be discussed in Sections \ref{sec.no.odd} and \ref{sec.no.simplex}. 

    However, it turns out that the rational Urysohn space $\Q\U$ does \emph{not} satisfy $\sigma$-SAP. Indeed, there exists an infinite subset $A\subseteq\dom(\Q\U)$ and a point $x\in \dom(\Q\U)\setminus A$ such that $d(A,x)=0$. Let $\A$ be the structure generated by $A$ and let $\B_1=\B_2$ be the structure generated by $A\sqcup\{x\}$. Then any amalgamation $\mathbf{C}$ would identify $\B_1\setminus\A$ and $\B_2\setminus\A$.
\end{remark}

\begin{remark}
    $\sigma$-SAP is closely related to \emph{uniform non-algebraicity} introduced by Tsankov in \cite{tsankov2024nonsingularprobabilitymeasurepreservingactions} to obtain de Finetti's style results. 
\end{remark}

Recall that a structure $\M$ is \textit{ultrahomogeneous} if every isomorphism between finite substructures of $\M$ extends to an automorphism of $\M$.

\setcounter{claim}{0}
\begin{proposition}\label{prop:SAOCSimpliesInvequalTail}
Let $\M$ be a countable relational ultrahomogeneous structure. Assume that $\M$ satisfies $\sigma$-SAP and that $\Aut(\M)$ weakly eliminates imaginaries. Then for every unitary representation $\pi\colon\Aut(\M)\to\UU(\HH)$ and every finite subset $A\subseteq \dom(\M)$, the subspaces $\HH_A$ and $\TT_A$ coincide.
\end{proposition}

\begin{proof}
    Let $G\coloneqq \Aut(\M)$ and $\Omega\coloneqq\dom(\M)$. Fix a unitary representation $\pi \colon G\to\UU(\HH)$ and a finite subset $A\subseteq\Omega$. The inclusion $\HH_A\subseteq\TT_A$ always holds, so let us prove the reverse inclusion. We start with the following claim.  
    
    \begin{claim}\label{claim.TT_A}
    $\displaystyle\TT_A=\overline{\bigcup_{\substack{B\subseteq\Omega \\\text{ finite }}}\TT_A\cap\HH_B}$. 
    \end{claim}
    \begin{cproof}
        Recall that the union of the $\HH_B$'s for $B\subseteq\Omega$ finite is dense in $\HH$ \cite[Lem.~2.6]{BJJ24}. Let $\xi\in\TT_A$. Then there exists a sequence of finite subsets $(B_n)_{n\geq 0}$ and a sequence of vectors $(\xi_n)_{n\geq 0}$ such that $\xi_n\in\HH_{B_n}$ and $\xi_n\to\xi$. Since $\HH_{B_n}\subseteq\HH_{A\cup B_n}$, we may assume that $A\subseteq B_n$ for every $n\geq 0$. Define $\eta_n\coloneqq p_{\TT_A}\xi_n$. Clearly, $\eta_n\in\TT_A$ and $\eta_n\to \xi$. So it suffices to prove that $\eta_n\in\HH_{B_n}$. But for $g\in G_{B_n}$, one has $\pi(g)\TT_A=\TT_A$ (because $G_{B_n}\subseteq G_A$) and therefore \[\pi(g)\eta_n=p_{\pi(g)\TT_A}\pi(g)\xi_n=p_{\TT_A}\xi_n=\eta_n.\qedhere\]
    \end{cproof}

    \begin{claim}\label{claim.TT_AcapHH_B}
        For every finite subset $B\subseteq\Omega$, we have $\TT_A\cap\HH_B\subseteq\HH_A$. 
    \end{claim}
    \begin{cproof}
        Without loss of generality, one may assume that $A\subseteq B$. Define $B_1\coloneqq B\setminus A$. Let $\xi\in\TT_A\cap\HH_B$. Let $C$ be a cofinite subset disjoint from $B$. Since $\xi$ belongs to $\TT_A \subseteq \HH_C$, there exists an increasing sequence $(C_n)_{n\geq 0}$ of finite subsets of $C$ disjoint from $A$ and vectors $\xi_n\in\HH_{A\sqcup C_n}$ such that $\xi_n\to \xi$. Up to enlarging the $C_n$'s, one may assume that $\bigcup_{n\geq 0} C_n=C$. Since $\M$ satisfies $\sigma$-SAP, there exists an amalgamation of $C\cup B$ with itself over $C\sqcup A$. By ultrahomogeneity, one readily gets the following. 
        
        \begin{fact*}
            There exists $B_2$ finite in $\Omega$ and an increasing family $(C_n')_{n\in \N}$ of finite subsets of $\Omega$ such that:
            \begin{itemize}
                \item there is $g_n\in G_{A\sqcup B_1}$ sending $C_n$ to $C_n'$,
                \item there is $h_n \in G_{A \sqcup C_n'}$ sending $B_1$ to $B_2$.
            \end{itemize}
        \end{fact*}

        Let us prove that $\xi\in \HH_{A\sqcup B_2}$. 
        We write
        \begin{align*}\lVert \xi-h_n\xi\rVert &\leq \lVert \xi-h_ng_n\xi_n\rVert + \lVert h_ng_n\xi_n-h_n\xi\rVert \\
        &\leq 2\lVert \xi-g_n\xi_n\rVert \\
        &= 2\lVert \xi-\xi_n\rVert \to 0
        \end{align*}

Thus, since for every $n$, $h_n\xi$ belongs to $\HH_{A\sqcup B_2}$, which is closed, we get that $\xi$ itself belongs to $\HH_{A\sqcup B_2}$. 
 Finally, $\xi$ belongs to $\HH_{A\sqcup B_1}\cap\HH_{A\sqcup B_2}=\HH^{\langle G_{A\sqcup B_1},G_{A\sqcup B_2}\rangle}=\HH_A$ and this finishes the proof of the claim. 
        \end{cproof}
Combining Claim \ref{claim.TT_A} and \ref{claim.TT_AcapHH_B}, we get that $\TT_A\subseteq\HH_A$, which finishes the proof. 
\end{proof}

As a consequence of the results in this section, we obtain the following theorem.

\begin{theorem}
\label{thm:omega-SAP.dissociated}
    Let $\M$ be a countable ultrahomogeneous relational structure. If $\M$ weakly eliminates imaginaries and satisfies the strong cofinite amalgamation property over countable subsets, then $\Aut(\M)$ is dissociated.
\end{theorem}

\subsubsection{\texorpdfstring{$\aleph_0$-categoricity and $\sigma$-SAP}{ω-categoricity and σ-SAP}}

We now turn to a first application of Theorem \ref{thm:omega-SAP.dissociated} to $\aleph_0$-categorical ultrahomogeneous structures in order to recover dissociation for oligomorphic groups with weak elimination of imaginaries and no algebraicity (Theorem 3.4 of \cite{JT}). 

Let $\M$ be a countable relational structure. We say that two tuples $\x=(x_1,\dots,x_n)$ and $\y=(y_1,\dots,y_n)$ on $\M$ have the same \textit{isomorphism type} if the map $x_i\mapsto y_i$ extends to an isomorphism between the substructures induced by $\x$ and $\y$. This defines an equivalence relation on tuples and we denote by $\tp(\x)$ the equivalence class of $\x$. A \textit{$k$-type} over $\M$ is simply the type of a $k$-tuple in $\M$. The \emph{tree of isomorphism types} of $\M$ is a rooted tree defined as follows: 
\begin{itemize}
    \item its vertex set is the collection of all $k$-types over $\M$, for every $k\geq 0$,
    \item for every $k\geq 0$, there is an edge between a $(k+1)$-type $t$ and a $k$-type $s$ if and only if there exists $(x_1,\dots,x_{k+1})\in\M^{k+1}$ such that $t=\tp(x_1,\dots,x_{k+1})$ and $s=\tp(x_1,\dots,x_k)$.
\end{itemize}

 The following result is a direct reformulation of Ryll-Nardzewski's theorem. 
 
\begin{lemma} \label{Lem.Ryll.Nardzewski}
    The boundary of the tree of isomorphism types of $\M$ is compact if and only if $\M$ is $\aleph_0$-categorical.
\end{lemma}

We now show that $\aleph_0$-categorical Fraïssé limits satisfy $\sigma-$SAP and even more: the cofinitess assumption is not even necessary in this case.

\begin{lemma}\label{lem.omegacatimpliesSAPAS}
    Let $\M$ be a countable relational ultrahomogeneous structure which has the strong amalgamation property. If $\M$ is $\aleph_0$-categorical, then $\M$ has the strong amalgamation property over arbitrary subsets. 
\end{lemma}

\begin{proof} 
Let $\A,\B_1,\B_2\in\sigma\mathrm{Age}(\M)$ and embeddings $f_i\colon\A\hookrightarrow \B_i$ satisfying $\lvert \dom(\B_i)\setminus \dom(f_i(\A))\rvert <+\infty$. If $\A$ is finite, there is nothing to prove since $\M$ has the strong amalgamation property. So let us assume that $\A$ is infinite. Fix an enumeration $a_1,\dots,a_n,\dots$ of $\dom(\A)$ and let $\A_n$ be the structure generated by $\{a_1,\dots,a_n\}$. Let $\x=(x_1,\dots,x_k)$ be an enumeration of $\dom(\B_1)\setminus \dom(f_1(\A))$ and $\y=(y_1,\dots,y_l)$ be an enumeration of $\dom(\B_2)\setminus \dom(f_2(\A))$.  Let $\mathbf{C}_n$ be an amalgamation of $\B_1\setminus f_1(\A\setminus \A_n)$ and $\B_2\setminus f_2(\A\setminus \A_n)$ over $\A_n$ as in Definition \ref{defn:SAPAS}. We identify the domain of $\mathbf{C}_n$ with the union of the domains of $\B_1\setminus f_1(\A\setminus \A_n)$ and $\B_2\setminus f_2(\A\setminus \A_n)$, and call  $t_n$ the type induced by $\mathbf{C}_n$ on $(\x,\y,\A_n)$. Let $\xi_n$ be any element of the boundary of the tree of types extending $t_n$. By compactness (Lemma \ref{Lem.Ryll.Nardzewski}), $(\xi_n)$ admits a subsequence converging to some point  $\xi$. Any infinite tuple $C=(c_1,c_2,\dots)\in\M^\N$ satisfying $(\tp(c_1,\dots,c_n))_{n\geq 0}=\xi$ is as wanted. Existence of such a tuple is ensured by ultrahomogeneity of $\M$. 
\end{proof}

As a consequence, we recover the result {\cite[Thm.~3.4]{JT}} of the second author and Tsankov, whose original proof uses the classification of unitary representations for oligomorphic groups due to Tsankov \cite{Tsankov}. Recall that a closed permutation group $G\leq\Sym(\Omega)$ is \textit{oligomorphic} if for every $n\in \N$, the diagonal action $G\curvearrowright\Omega^n$ has only finitely many orbits.

\begin{theorem}\label{thm.oligomorphic.dissociated}
Let $G\leq \Sym(\Omega)$ be a closed subgroup. Assume that $G$ is oligomorphic, has no algebraicity, and weakly eliminates imaginaries. Then $G$ is dissociated. 
\end{theorem}

\begin{proof}
Let $\M_G$ be the canonical structure associated with $G$ (for a definition, see Section 2.3 of \cite{Cameron}). It is a countable relational ultrahomogeneous structure, which satisfies the strong amalgamation property. Moreover, by the Ryll-Nardzewski Theorem, $\M_G$ is $\aleph_0$-categorical (see \cite[Th.~2.10]{Cameron}). Thus by Lemma \ref{lem.omegacatimpliesSAPAS}, $\M_G$ satisfies $\sigma$-SAP. Finally, Theorem \ref{thm:omega-SAP.dissociated} allows us to conclude that $\Aut(\M_G)=G$ is dissociated.
\end{proof}

\section{New examples of dissociated permutation groups}\label{sec.examples}

\subsection{Weak elimination of imaginaries from Ramsey theorems}\label{sec.RamseyWEI}

This short subsection is here to provide the tools we need in the rest of the section to prove that the described structures admit weak elimination of imaginaries. Simply put, we prove that in the right context, weak elimination of imaginaries is implied by a Ramsey property.

Let $\mathcal{F}$ be a class of structures. For $A,B\in \mathcal{F}$, denote by $\binom B A$ the set of embeddings of $A$ in $B$. We say that $\mathcal{F}$ has the \emph{Ramsey Property} if for all $A,B\in \mathcal{F}$ and $k\in \N$, there is a $C\in \mathcal{F}$ such that for every $\chi \colon {\binom C A}\to \{1,\dots, k\}$, there exists $\phi \in\binom{C}{B}$ such that all embeddings of $A$ with image in $\phi(B)$ have the same image under $\chi$.

Adopting flexible notations, we will say two structures of the form $(C,x_1,\dots,x_n)$ and $(C,y_1,\dots,y_n)$ are isomorphic, and write $(C,x_1,\dots,x_n) \simeq (C,y_1,\dots,y_n)$, if the map $C\cup \{x_1,\dots, x_n\}\to C\cup \{y_1,\dots, y_n\}$ that fixes $C$ and sends  $x_i$ to $y_i$ is an isomorphism. We say that $\mathcal{F}$ is $2$-symmetric if the following holds:
$$\forall (C,x,y)\in \mathcal{F}, \ \left[ (C,x) \simeq (C,y) \Rightarrow (C,x,y)\simeq (C,y,x)\right].$$
For a class $\mathcal{F}$, we denote by $\mathcal{F}_<$ the class of all ordered elements of $\mathcal{F}$, that is of all expansions of elements of $\mathcal{F}$ by a linear order. If $\mathcal{F}$ is a Fraïssé class, we will denote its Fraïssé limit by $\mathbb{F}$.

\begin{proposition}\label{prop.RamseyWEI}
    Let $\mathcal{F}$ be a $2$-symmetric Fra\"iss\'e class with the strong amalgamation property. If $\mathcal{F}_<$ has the Ramsey property, then $\mathbb{F}$ weakly eliminates imaginaries.
\end{proposition}

\begin{proof}
    Let us first give a characterization of weak elimination of imaginaries, adapted from \cite[Cor.~7.6]{Conant}.
    \begin{fact}\label{fact}
        Let $\mathcal{F}$ be a Fra\"issé class and $\mathbb{F}$ its Fraïssé limit such that:
        \begin{enumerate}[label=(\roman*)]
            \item\label{fact.item1} $\mathcal{F}$ has the strong amalgamation property,
            \item\label{fact.item2} for every finite $D\subseteq\mathbb{F}$ and every $x\in \mathbb F$, there is no non-trivial $\mathrm{Aut}(\mathbb F)_{D}$-invariant equivalence relation on $\{x'\in \mathbb{F} \colon \ (D,x)\simeq (D,x')\}$.
        \end{enumerate}
      Then $\mathrm{Aut}(\mathbb{F})$ weakly eliminates imaginaries.
    \end{fact}

    \begin{proof}[Proof of the fact]

        Write $G\coloneqq \Aut(\mathbb{F})$ and let $V\leqslant G$ be an open subgroup. There exists a finite subset $A\subseteq \Omega$ such that $G_A\leqslant V$. Write $A=\{a_1,\dots, a_n\}$ in a injective way and let $\overline{a} =(a_1,\dots,a_n)$. Let be $\sim$ the $G$-invariant equivalence relation on $G\cdot \overline{a}$ defined by $g\cdot \overline{a} \sim h\cdot\overline{a}$ if $g^{-1}h\in V$. Clearly, the stabilizer in $G$ of the $\sim$-class of $\overline{a}$ is precisely $V$. Proposition 5.5.3 in \cite{kruckman_infinitary_2016} yields $I\subset \{1,\dots,n\}$ and a subgroup $P$ of $\mathrm{Sym}(I)$ such that for any $\overline{b}\in G\cdot \overline{a}$, we have
        \begin{equation}\label{Eq.WEI}
        \overline{a}\sim \overline{b} \Leftrightarrow \bigvee_{f\in P} \bigwedge_{i\in I} a_i=b_{f(i)}.
        \end{equation}

        We set $A'=\{a_i \ \colon \ i\in I\}$ and prove that $G_{A'}$ is a finite-index subgroup of $V$. Since $G_{A'}$ clearly preserves the $\sim$-class of $\overline{a}$, it is indeed a subgroup of $V$. Moreover, $V$ acts on $I$ by $(\ref{Eq.WEI})$ hence the existence of a morphism from $V$ to the finite group $\mathrm{Sym}(I)$. The kernel of this action being precisely $G_{A'}$, the quotient $V/G_{A'}$ is indeed finite.
    \end{proof}


    Write $G\coloneqq \Aut(\mathbb{F})$. Towards a contradiction, let us assume that $\mathcal{F}$ does not satisfy \ref{fact.item2}. In this case, there exist $(D,x)\subseteq \mathbb{F}$ finite and a non-trivial $G_D$-invariant equivalence relation $\sim$ on $X\coloneqq \{x'' \in \mathbb{F} \colon \ (D,x)\simeq (D,x'')\}$. By non-triviality, there exist $x',y\in X\backslash\{x\}$ such that $x\sim x' \nsim y$. By the no-algebraicity hypothesis, $G_{D,x,y}\cdot x'$ is infinite, and we can pick $z\neq z' \in G_{D,x,y}\cdot x'\backslash \{x'\}$. In particular, $(D,y,z)\simeq(D,y,z')$ and $z\sim z' \sim x$.

    We now produce a counterexample to the Ramsey property for $\mathcal{F}_<$. Fixing a relation symbol $<$ for the added linear ordering, we will denote by $<^C$ its interpretation in a given structure $C\in \mathcal{F}_<$. Let $A$ be any expansion of $(D,y,z)$ by a linear order where $y <^A z$ and $D<^A y$. Similarly, let $B$ any expansion of $(D,z',y,z)$ by a linear order where $z'<^By<^B z$ and such that, as ordered expansions, $(D,z',y)\simeq A$. Note that by $2$-symmetry $B$ contains two copies of $A$: $(D,y,z)$ and $(D,z',y)$. For every $C\in \mathcal{F}_<$, we will produce a bad coloring of $\binom C A$.

    Let us fix $C\in \mathcal{F}_<$. For every copy $D_0$ of $D$ in $C$, we denote by $\sim_{D_0}$ the $G_{D_0}$-invariant equivalence relation on $X_{D_0}\coloneqq\{a\in C \colon \ (D_0,a)\simeq (D,x)\}$ corresponding to $\sim$. We order $X_{D_0}$ in such a way that the classes form intervals, and denote by $\prec_{D_0}$ this ordering. To color a copy $A_1=(D_1,x_1,y_1)$ of $A$ in $C$, compare the ordering of $x_1$ and $y_1$ in $<^C$ and $\prec_{D_1}$. Associate $0$ to $A_1$ if they agree and $1$ otherwise. 
    
    Let us consider $B_1= (D_1,z_1',y_1,z_1)$ a copy of $B$ inside $C$. Note that since $\mathcal{F}$ is $2$-symmetric, $B_1$ contains two copies of $A$: $A_1\coloneqq(D_1,y_1,z_1)$ and $A_2\coloneqq(D_1,z_1',y_1)$. In the described coloring, those two copies of $A$ can never have the same color. Indeed, $z_1\prec_{D_1} y_1$ iff $z_1'\prec_{D_1} y_1 $ while $z_1'<^{C}y_1<^{C}z_1$.
\end{proof}

\subsection{Metrically homogeneous spaces}

In this section, we consider some classes of countable metrically homogeneous spaces of infinite diameter for which the isometry group satisfies the assumptions of Theorem \ref{thm:omega-SAP.dissociated}. A metric space $(X,d)$ is \emph{metrically homogeneous} if every surjective isometry between two finite subsets of $X$ extends to a surjective isometry of $X$.




An \emph{integral metric space} is a metric space $(X,d)$ such that $d(X\times X)\subseteq\N$. Every countable integral metric space $(X,d)$ may be considered as a countable relational structure in the language $\LL=(R_n)_{n\geq 0}$ where each $R_n$ is a relation of arity $2$ whose interpretation $R_n^X$ in $X$ is given by
\[R_n^X\coloneqq\{(x,y)\in X^2\colon d(x,y)=n\}.\]
Notice that $X$ is ultrahomogeneous (as a countable structure) if and only if $X$ is metrically homogeneous. Notice that $\mathrm{Isom}(X,d)=\mathrm{Aut}(X)$. For integral metric spaces, there is a canonical amalgamation over non-empty subspaces that will be useful later.


\begin{definition}
    Let $(X_1,d_1)$ and $(X_2,d_2)$ be two integral metric spaces such that the intersection $Y\coloneqq X_1\cap X_2$ is non-empty. Assume that $d_1$ and $d_2$ coincide on $Y\times Y$. Let $X\coloneqq X_1\cup X_2$ and define a map $d\colon X\times X\to\N$, which restricts to $d_1$ on $X_1\times X_1$, to $d_2$ on $X_2\times X_2$ and such that for all $x_1\in X_1\setminus Y$ and $x_2\in X_2\setminus Y$,
    \[d(x_1,x_2)=\min\{d(x_1,y)+d(y,x_2)\colon y\in Y\}.\]
    Then $(X,d)$ is an integral metric space. It is called the \emph{metric amalgam} of $(X_1,d_1)$ and $(X_2,d_2)$ over $Y$ and is denoted by $(X_1,d_1)*_Y(X_2,d_2)$. 
    \end{definition}

\subsubsection{Integral metric spaces with no small triangle of odd perimeter}\label{sec.no.odd}

Let $p\geq 1$ be an integer. A metric space $(X,d)$ has \emph{no triangle of odd perimeter less than $p$} if the following holds: for every $(x,y,z)\in X^3$, if $d(x,y)+d(y,z)+d(z,x)$ is odd, then $d(x,y)+d(y,z)+d(z,x)>p$. 

\begin{lemma}\label{lem.amalgam.no.triangle}
    Let $p\geq 1$ be an integer. Let $(X_1,d_1)$ and $(X_2,d_2)$ be two integral metric spaces which contain no triangle of odd perimeter less than $p$. If $X_1\cap X_2$ is non-empty, then the metric amalgam of $(X_1,d_1)$ and $(X_2,d_2)$ over ${X_1\cap X_2}$ contains no triangle of odd perimeter less than $p$.
\end{lemma}

\begin{proof} 
    Let $X\coloneqq X_1\cup X_2$ and $Y\coloneqq X_1\cap X_2$ . Fix $x_1\in X_1\setminus Y$, $x_2\in X_2\setminus Y$ and $x\in X$ such that  $d(x_1,x_2)+d(x_2,x)+d(x,x_1)$ is odd. There are two cases to check. 
    
    Assume that $x\in Y$. Let $y\in Y$ be such that $d(x_1,x_2)=d(x_1,y)+d(y,x_2)$. Then $d(x_1,y)+d(y,x_2)+d(x_2,x)+d(x,x_1)$ is odd. Therefore, there is $i\in\{1,2\}$ such that $d(x_i,y)+d(y,x)+d(x,x_i)$ is odd. But this last quantity is $>p$ since $(x_i,y,x)\in X_i\times X_i\times X_i$. Therefore, 
    \begin{align*}
        d(x_1,x_2)+d(x_2,x)+d(x,x_1) & = d(x_1,y)+d(y,x_2)+ d(x_2,x)+d(x,x_1) \\
&\geq d(x_i,y)+d(y,x)+d(x,x_i) \\
& >p,
    \end{align*}which is what we wanted.

    Assume now that $x\in X\setminus Y$. Without loss of generality, we may assume that $x\in X_1\setminus Y$. Let $y,z\in Y$ be such that $d(x_1,x_2)=d(x_1,y)+d(y,x_2)$ and $d(x_2,x)=d(x_2,z)+d(z,x)$. Then $d(x_1,y)+d(y,x_2)+d(x_2,z)+d(z,x)+d(x,x_1)$ is odd. Therefore, among the following three quantities 
    \begin{align*}
        c_1&\coloneqq d(x,x_1)+d(x_1,y)+d(y,x),\\
        c_2&\coloneqq d(x,y)+d(y,z)+d(z,x),\\
        c_3&\coloneqq d(z,y)+d(y,x_2)+d(x_2,z),
    \end{align*}
    at least one of them is odd. But if $c_i$ is odd, then $c_i>p$. By the triangle inequality, we obtain that
    \[d(x_1,x_2)+d(x_2,x)+d(x,x_1)\geq \max(c_1,c_2,c_3) >p,\]
    which concludes the proof. 
\end{proof}

\begin{lemma}\label{lem.no.odd.triangle.Fraisse}
    Let $p\geq 1$ be an integer. Then the class of finite integral metric spaces with no triangle of odd perimeter less than $p$ is a Fraïssé class satisfying the strong amalgamation property. 
\end{lemma}

\begin{proof}
     It is clear that this class satisfies the hereditary property. The joint embedding property will follow from the amalgamation property since the language is relational and there is no constant relation. Let us show the strong amalgamation property. For this, fix three finite integral metric spaces $(A,d_A), (B_1,d_1)$ and $(B_2,d_2)$ with no triangle of odd perimeter less than $p$. Assume that $(B_1,d_1)$ and $(B_2,d_2)$ contain an isometric copy of $A$. Let $C$ be the set obtained by identifying both copies of $A$ in the disjoint union $B_1\sqcup B_2$. We define a new metric $d$ on $C$ as follows. There are two cases to check.
     
     If $A$ is non-empty, then let $d$ be the metric amalgam of $(B_1,d_1)$ and $(B_2,d_2)$ over $(A,d_A)$. Then $(C,d)$ is a finite integral metric space with no triangle of odd perimeter less than $p$ by Lemma \ref{lem.amalgam.no.triangle}. 
     
     If $A$ is empty, then let $d$ be the metric on $C$ which restricts to $d_i$ on $X_i\times X_i$ and such that for all $x_1\in B_1$, $x_2\in B_2$, 
     \[d(x_1,x_2)=\max(p,\mathrm{diam}(B_1,d_1),\mathrm{diam}(B_2,d_2)).\]
     Then $(C,d)$ is an integral metric space, with no triangle of odd perimeter less than $p$. 
     
     In both cases, we have found a strong amalgam of $(B_1,d_1)$ and $(B_2,d_2)$ over $(A,d_A)$, which concludes the proof. 
\end{proof}

In the sequel, given an integer $p\geq 1$, we denote by $\M_p$ the Fraïssé limit of the class of finite integral metric spaces with no triangle of odd perimeter less than $p$. Notice that when $p=1$, $\M_p$ is the integral Urysohn space $\Z\U$. 

\begin{lemma}\label{lem.no.odd.omega.SAP}
    $\M_p$ satisfies $\sigma$-SAP.  
\end{lemma}

\begin{proof}
    Since the Fraïssé class of finite integral metric spaces with no triangle of odd perimeter less than $p$ satisfies the strong amalgamation property, it remains to prove the existence of a strong cofinite amalgam over a countable metric space. The existence of such an amalgam is provided by Lemma \ref{lem.amalgam.no.triangle}.
\end{proof} 

We postpone the proof of the last ingredient to Section \ref{sec.RamseyWEI}.

\begin{lemma}
$\Isom(\M_p)$ weakly eliminates imaginaries.
\end{lemma}

\begin{proof}
    The class $\mathcal{F}$ of finite integral metric spaces with no triangle of odd perimeter less than $p$ is clearly $2$-symmetric and has the strong amalgamation property by Lemma \ref{lem.no.simplex.strong.amalgamation}. Moreover, $\mathcal{F}_{<}$ has the Ramsey property by \cite[Thm.~8.2]{RamseyExpansions}. So Proposition \ref{prop.RamseyWEI} yields the desired conclusion. 
\end{proof}

As a consequence of the results obtained in this section, Theorem \ref{thm:omega-SAP.dissociated} applies. 

\begin{theorem}
   $\Isom(\M_p)$ is dissociated. 
\end{theorem}

We close this section by discussing the topological and geometric nature of $\Isom(\M_p)$. 

\begin{proposition}
    $\Isom(\M_p)$ is locally Roelcke-precompact but not coarsely bounded. 
\end{proposition}

\begin{proof}
    Let $X$ be the distance-1 graph associated with $\M_p$. It is a straightforward consequence of homogeneity and universality that $\Aut(X) =\Isom(\M_p)$. Thus, \cite[Th.~3.5]{Rosendal} applies and $\Isom(M,d)$ is locally Roelcke-precompact. However, it is not bounded because it acts transitively by isometries on the unbounded space $\M_p$.
\end{proof}

Finally, we can show that this group satisfies the Howe-Moore property. 

\begin{proposition}
$\Isom(\M_p)$ has the Howe-Moore property with respect to coarsely bounded sets.
\end{proposition}

\begin{proof}
    The proof of \cite[Thm.~3.5]{Rosendal} shows that for every $a\in\dom(\M_p)$, the point stabilizer $\Isom(\M_p)_a$ is Roelcke-precompact and therefore coarsely bounded. Since $\Isom(\M_p)$ is dissociated, Theorem \ref{thm.Howe.Moore.dissociated} implies that $\Isom(\M_p)$ has the Howe-Moore property with respect to coarsely bounded sets. 
\end{proof}

\subsubsection{Integral metric spaces with no unit simplex}\label{sec.no.simplex}

Let $r\geq 3$ be an integer. A unit $(r-1)$ simplex in an integral metric space $(X,d)$ is a subset of $X$ consisting of $r$ points that lie mutually at distance $1$. 

\begin{lemma}\label{lem.amalgam.no.simplex}
    Let $r\geq 3$ be an integer. Let $(X_1,d_1)$ and $(X_2,d_2)$ be two integral metric spaces that contain no unit $(r-1)$-simplex. If $X_1\cap X_2$ is non-empty, then the metric amalgam of $(X_1,d_1)$ and $(X_2,d_2)$ over $X_1\cap X_2$ contains no unit $(r-1)$-simplex.
\end{lemma}

\begin{proof}
    Let $Y\coloneqq X_1\cap X_2$ and let $X\coloneq (X_1,d_1)*_Y(X_2,d_2)$ be the metric amalgam. Let $F\subseteq X$ be a subset consisting of $r$ points. If $F\subseteq X_1$ or $F\subseteq X_2$, then $F$ is not a unit $(r-1)$ simplex by assumption. Else, there exist two points $x_1,x_2\in F$ with $x_1\in X_1\setminus Y$ and $x_2\in X_2\setminus Y$. Let $y\in Y$ be such that $d(x_1,x_2)=d(x_1,y)+d(y,x_2)$. Then $d(x_1,x_2)>1$ and thus $F$ is not a unit $(r-1)$ simplex. 
\end{proof}

\begin{lemma}\label{lem.no.simplex.strong.amalgamation}
    Let $r\geq 3$ be an integer. Then the class of finite integral metric spaces with no unit $(r-1)$-simplex is a Fraïssé class satisfying the strong amalgamation property.
\end{lemma}

\begin{proof}
    The proof is identical to that of Lemma \ref{lem.no.odd.triangle.Fraisse}. 
\end{proof}

In the sequel, given an integer $r\geq 3$, we denote by $\mathbf{N}_r$ the Fraïssé limit of the class of finite integral metric spaces with no unit $(r-1)$-simple.

\begin{lemma}
   $\mathbf{N}_r$ satisfies $\sigma$-SAP. 
\end{lemma}

\begin{proof}
 As in Lemma \ref{lem.no.odd.omega.SAP}, the proof follows from the fact that the metric amalgam over a non-empty subset preserves the property of having no unit $(r-1)$-simplex, see Lemma \ref{lem.amalgam.no.simplex}. 
\end{proof}

Again, the last ingredient will be provided by the results in Section \ref{sec.RamseyWEI}.

\begin{lemma}
    $\Isom(\mathbf{N}_r)$ weakly eliminates imaginaries. 
\end{lemma}

\begin{proof}
    Again, the class $\mathcal{F}$ of finite integral metric spaces with no unit $(r-1)$-simple is $2$-symmetric has the strong amalgamation property by Lemma \ref{lem.no.simplex.strong.amalgamation}. Moreover, $\mathcal{F}_{<}$ has the Ramsey property by \cite[Thm.~8.2]{RamseyExpansions}. So Proposition \ref{prop.RamseyWEI} yields the desired conclusion. 
\end{proof}

Therefore, Theorem \ref{thm:omega-SAP.dissociated} applies. 

\begin{theorem}\label{thm.examples.metrically.homogeneous.dissociated}
     $\Isom(\mathbf{N}_r)$ is dissociated. 
\end{theorem}

As in the previous section, we obtain the following two results. Since the proofs are identical, we omit them. 

\begin{proposition}
    $\Isom(\mathbf{N}_r)$ is locally Roelcke-precompact but not coarsely bounded. 
\end{proposition}

\begin{proposition}
    $\Isom(\mathbf{N}_r)$ has the Howe-Moore property with respect to coarsely bounded sets. 
\end{proposition}

\subsection{Diversities}

We now consider a generalization of the concept of metric space. Given a set $X$, we denote by $\PP_{\mathrm{fin}}(X)$ the set of finite subsets of $X$. A \emph{diversity} is a couple $(X,\delta)$ where $X$ is a set and $\delta \colon \PP_{\mathrm{fin}}(X)\to\R_+$ is a map satisfying:
\begin{itemize}
    \item $\delta(A)=0$ if and only if $\lvert A\rvert \leq 1$,
    \item if $B\neq \emptyset$, then $\delta(A\cup C)\leq \delta(A\cup B)+\delta(B\cup C)$. 
\end{itemize}
We call the map $\delta$ a \textit{diversity map}. Note that a metric space always defines a diversity, by seeing the diameter as a diversity map. Conversely, a diversity also defines a metric space but contains more information.

The above axioms imply that a diversity map $\delta$ is monotonous: if $A\subseteq B$, then $\delta(A)\leq \delta(B)$. Moreover, it is sublinear on sets with non-empty intersection: if $A\cap B\neq\emptyset$, then $\delta(A\cup B)\leq \delta(A)+\delta(B)$. We say that a diversity is integral (resp. rational) if the diversity map takes only integral (resp. rational) values. In this section, we will discuss diversities constructed by means of a Fraïssé limit. We first focus on diversities with integral values. 
\begin{lemma}
    The class of finite  diversities whose diversity map takes only integral values is a Fraïssé class, the limit of which we denote by $\Z\D$. 
\end{lemma}

We refer to \cite[Prop.~3.10]{Hallbäck} for a proof of this lemma. The only non-obvious part is the amalgamation property. As for integral metric spaces, it turns out that there is a natural way of  amalgamating integral diversities (of any cardinality). Let $X$ be a set. A collection $E_1,\dots, E_n$ of finite subsets of $X$ is \emph{connected} if the intersection graph associated with the $E_i$'s is connected. 

\begin{definition}\label{def.amalgam.diversities}
    Let $(X_1,\delta_1)$ and $(X_2,\delta_2)$ be two integral diversities such that the intersection $Y\coloneqq X_1\cap X_2$ is non-empty. Assume that $d_1$ and $d_2$ coincide on $\PP_{\mathrm{fin}}(Y)$. Let $X=X_1\cup X_2$ and define a map $\delta \colon \PP_{\mathrm{fin}}(X)\to\N$ as follows:
    \[\delta(A)\coloneqq\min\{\sum_{i=1}^n\delta_{k_i}(E_i)\colon E_1,\dots,E_n\text{ is connected}, A\subseteq\bigcup_{i=1}^nE_i\text{ and }E_i\subseteq X_{k_i}\}.\]
    Then $(X,\delta)$ is an integral diversity such that the restriction of $\delta$ to $\PP_{\mathrm{fin}}(X_i)\subseteq\PP_{\mathrm{fin}}(X)$ coincides with $\delta_i$ for every $i\in\{1,2\}$. It is called the \emph{diversity amalgam} of $(X_1,\delta_1)$ and $(X_2,\delta_2)$ over $Y$. 
\end{definition}

As before, we use this amalgamation to obtain $\sigma$-SAP.

\begin{lemma}
    $\Z\D$ satisfies $\sigma$-SAP. 
\end{lemma}

\begin{proof}
This is a consequence of the existence of the amalgam in Definition \ref{def.amalgam.diversities}
\end{proof}

\begin{remark}
    If $\Q\D$ denotes the Fraïssé limit of finite diversities whose diversity map takes only rational values, then $\Q\D$ does \emph{not} satisfy $\sigma$-SAP. The reason is the same as the one discussed in Remark \ref{rem.omega.SAP.examples} for the rational Urysohn space $\Q\U$. 
\end{remark}

Again, weak elimination of imaginaries is technical, but the results of Section \ref{sec.RamseyWEI} allow us to prove that it holds for $\Z\D$ and $\Q\D$.

\begin{lemma}
    $\Aut(\Z\D)$ and $\Aut(\Q\D)$ weakly eliminate imaginaries.
\end{lemma}

\begin{proof}
    Apply Proposition \ref{prop.RamseyWEI}, the hypotheses of which are satisfied by Example 7.9 and Theorem 6.15 of \cite{HubickaKonecny}.
\end{proof}

Therefore, Theorem \ref{thm:omega-SAP.dissociated} applies to $\Z\D$.

\begin{theorem}\label{thm.Aut(ZD).dissociated}
    $\Aut(\Z\D)$ is dissociated. 
\end{theorem}

Recall that we denote by $\Q\D$ the Fraïssé limit of finite diversities whose diversity map only takes rational values. Even though $\Q\D$ does not satisfy $\sigma$-SAP, we can still obtain dissociation via approximating sequences. Indeed, by using the analogue of Kat\v{e}tov functions for diversities developed in \cite{BryantNiesTupper}, one can build an approximating sequence for $\Q\D$ consisting of dissociated groups and therefore obtain the following result by Theorem \ref{thm.approx.sequence}.

\begin{theorem}\label{thm.Aut(QD).dissociated}
    $\Aut(\Q\D)$ is dissociated. 
\end{theorem}

\begin{proposition}
    $\Aut(\Z\D)$ and $\Aut(\Q\D)$ have the Howe-Moore property with respect to coarsely bounded sets. 
\end{proposition}

\begin{proof}
    Let $G$ be either $\Aut(\Z\D)$ or $\Aut(\Q\D)$. Then $G$ is dissociated by Theorems \ref{thm.Aut(ZD).dissociated} and \ref{thm.Aut(QD).dissociated}. We leave to the cautious reader the proof of the fact that the amalgamation of finite diversities (the above definition applies to \textit{finite} diversities, even if the diversity map takes non-integer values) is a functorial amalgamation in the sense of Rosendal \cite[Def.~6.30]{Rosendal}. Therefore, by Theorem 6.31 of \cite{Rosendal}, $G_A$ is coarsely bounded for every non-empty finite subset $A$ of either $\Z\D$ or $\Q\D$. Finally, Theorem \ref{thm.Howe.Moore.dissociated} allows us to conclude that $G$ has the Howe-Moore property with respect to coarsely bounded sets. 
\end{proof}

\subsection{Free amalgamation property}
The aim of this section is to prove that the automorphism group of a Fraïssé limit with the free amalgamation property is dissociated, by showing that such structures have $\sigma$-SAP. 

\begin{definition}
    A Fra\"iss\'e limit $\mathbf{M}$ in a relational language $\mathcal{L}$ satisfies the free amalgamation property (FAP) if for all $\A,\B_1,\B_2\in\mathrm{Age}(\M)$ and all embeddings $f_i\colon\A\hookrightarrow \B_i$, there exists $\mathbf{C}\in\mathrm{Age}(\M)$ and embeddings $g_i \colon \B_i\hookrightarrow \mathbf{C}$ such that   
    \begin{itemize}
        \item the diagram  
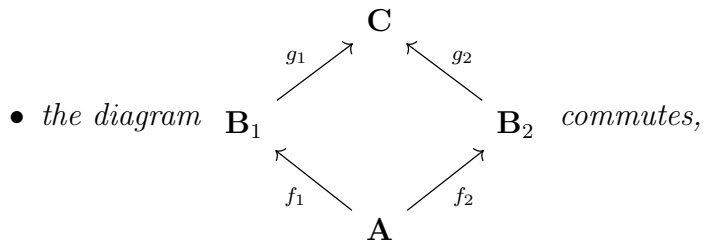
\begin{tikzcd}
                      & \mathbf{C}                                      &                        \\
\B_1 \arrow[ru, "g_1"] &                                        & \B_2 \arrow[lu, "g_2"'] \\
                      & \A \arrow[lu, "f_1"] \arrow[ru, "f_2"'] &                       
\end{tikzcd} commutes,

\item if there exist $b_i\in \B_i$ satisfying $g_1(b_1)=g_2(b_2)$, then there exists $a\in \A$ satisfying $f_i(a)=b_i$,
\item for every relation $R\in \mathcal{L}$ and every tuple $\overline{c}$ of elements in $\mathbf{C}$ such that $R^\mathbf{C}(\overline{c})$ holds, then $\overline{c}$ belongs either to $g_1(\B_1)$ or to $g_2(\B_2)$.
    \end{itemize}
\end{definition}

We refer to \cite[Example~2.3]{MacphersonTent} for some examples of Fraïssé limits with the free amalgamation property (FAP). Since all of them are $\aleph_0$-categorical (and therefore Roelcke-precompact), we give below an example of a Fraïssé limit satisfying (FAP) whose automorphism group is not Roelcke-precompact. 

\begin{example}
    Given a graph $\Gamma$, denote by $E(G)$ the set of its edges. Suppose $\Gamma$ is endowed with a map $c\colon E(\Gamma)\to\N$, thought of as a coloring of its edges. It can then be viewed as a structure in the language $\LL=(R_n)_{n\geq 0}$, where each $R_n$ is a relation of arity $2$ whose interpretation $R_n^\Gamma$ is given by \[R_n^\Gamma\coloneqq\{(x,y)\in V(\Gamma)^2\colon \{x,y\}\in E(\Gamma)\text{ and }c(\{x,y\})=n\}.\]
    Let $\CC$ be the class of all structures $(\Gamma,c)$ such that $\Gamma$ is a finite graph and $c\colon E(\Gamma)\to\N$. Then $\CC$ is a Fraïssé class with the free amalgamation property. Let $\M$ be its Fraïssé limit. Let us show that $\Aut(\M)$ is not Roelcke-precompact. Since $\M$ is ultrahomogeneous and one-point substructures of $\M$ are all isomorphic, then $\Aut(\M)$ acts transitively on its domain. However, there are countably many isomorphism types of two-point substructures in $\M$ which are given by the values of $c$. Therefore, there are infinitely many orbits for the action of $\Aut(\M)$ on pairs of points. Thus, the action of $\Aut(\M)$ on $\dom(\M)$ is not oligomorphic. This shows that $\Aut(\M)$ is not Roelcke-precompact by \cite[Thm.~2.4]{Tsankov}. 
\end{example}

However, the above example is coarsely bounded as a consequence of the following. 

\begin{lemma}
    Let $\M$ be a Fraïssé limit with the free amalgamation property. Then $\Aut(\M)$ is coarsely bounded.
\end{lemma}

\begin{proof}
    Let $\CC$ be the Fraïssé class whose limit is $\M$. It is straightforward to check that the free amalgamation provides a functorial amalgamation of $\CC$ over the empty structure $\emptyset$ in the sense of Rosendal \cite[Def.~6.30]{Rosendal}. Therefore, $\Aut(\M)$ is coarsely bounded by Theorem 6.31 of \cite{Rosendal}.
\end{proof}

\begin{lemma}\label{lem.free.noAlg.WEI}
    Let $\M$ be a Fraïssé limit with the free amalgamation property. Then $\Aut(\M)$ has no algebraicity and weakly eliminates imaginaries.
\end{lemma}

\begin{proof}
    This is the combination of Lemma 2.5 and 2.7 in \cite{MacphersonTent} and Lemma \ref{lem.caracterisation.noAlg.WEI}. 
\end{proof}

\begin{theorem}\label{thm.FAP.sigmaSAP}
    Let $\M$ be a Fraïssé limit satisfying the free amalgamation property. Then $\M$ satisfies $\sigma$-SAP.  
\end{theorem}

\begin{proof}
Fix $\A,\B_1,\B_2\in\sigma\mathrm{Age}(\M)$ and embeddings $f_i\colon\A\hookrightarrow\B_i$. Consider $\mathbf{C}$ the free amalgamation of $\B_1$ and $\B_2$ over $\A$, i.e. the union of $\B_1$ and $\B_2$ identifying $\A$ and where all the relations are the ones from $\B_1$ and $\B_2$. $\mathbf{C}$ is well-defined and is in $\sigma\mathrm{Age}(\M)$ as all of its finite substructures are in $\mathrm{Age}(\M)$.
\end{proof}

\begin{theorem}\label{thm.free.dissociated}
    If $\M$ is a Fraïssé limit satisfying the free amalgamation property, then $\Aut(\M)$ is dissociated. 
\end{theorem}

\begin{proof}
     Let $\M$ be a Fraïssé limit satisfying the free amalgamation property. Then $\M$ weakly eliminates imaginaries by Lemma \ref{lem.free.noAlg.WEI} and satisfies $\sigma$-SAP by Theorem \ref{thm.FAP.sigmaSAP}. Finally, Theorem \ref{thm:omega-SAP.dissociated} yields that $\Aut(\M)$ is dissociated. 
\end{proof}

\printbibliography

\Addresses
\end{document}

%% file: packages.tex
\usepackage[top=2cm, bottom=2cm, left=3cm , right=3cm]{geometry} 
\usepackage[utf8]{inputenc}
\usepackage[T1]{fontenc}
\usepackage{amssymb}
\usepackage{amsmath}
\usepackage{dsfont}
\usepackage{mathtools}
\usepackage{amsfonts}
\usepackage[greek, english]{babel}
\usepackage{amsthm}
\usepackage{tikz}
\usepackage[hidelinks]{hyperref}
\usepackage{enumitem}  
\usepackage{marginnote}
\usepackage[symbol]{footmisc}

\usepackage[draft,inline,nomargin,index]{fixme}

\fxsetup{theme=color,mode=multiuser}
\FXRegisterAuthor{cj}{acj}{\color{green}CJ}
\FXRegisterAuthor{mj}{amj}{\color{red}MJ}
\FXRegisterAuthor{rb}{arb}{\color{blue}RB}

\usepackage{ wasysym }

\usepackage{setspace}
\usepackage{tikz-cd} 

%% file: environnements.tex
\addto\captionsfrench{%
  }

\newenvironment{cproof}{\begin{proof}[Proof of the 
		claim]}{\end{proof}}
\newtheoremstyle{capitale}    
    {\topsep}                    
    {\topsep}                    
    {\itshape}                   
    {}                           
    {\scshape}                   
    {. ---}                          
    {.5em}                       
    {} 
    
\newtheoremstyle{remark}    
    {\topsep}                    
    {\topsep}                    
    {}                   
    {}                           
    {\scshape}                   
    {. ---}                          
    {.5em}                       
    {} 

\theoremstyle{capitale}
\newtheorem{theorem}{Theorem}[section]
\newtheorem*{theorem*}{Theorem}
\newtheorem{lemma}[theorem]{Lemma}
\newtheorem*{claim*}{Claim}
\newtheorem{claim}{Claim}
\newtheorem*{lemme*}{Lemme}
\newtheorem{corollary}[theorem]{Corollary}
\newtheorem{definition}[theorem]{Definition}
\newtheorem*{definition*}{Definition}
\newtheorem{fact}[theorem]{Fact}
\newtheorem*{fact*}{Fact}
\newtheorem*{notation}{Notation}

\newtheorem{proposition}[theorem]{Proposition}

\theoremstyle{remark}
\newtheorem{remark}[theorem]{Remark}

\newtheorem{example}[theorem]{Example}
\newtheorem{question}[theorem]{Question}

\usepackage[backend=biber, style=numeric, maxnames=50]{biblatex}
\usepackage{csquotes}
\usepackage{todonotes}
\DefineBibliographyStrings{english}{in = {}}

%% file: commandes.tex
\newcommand{\BB}{\mathcal{B}}

\newcommand{\PP}{\mathcal{P}}
\newcommand{\FF}{\mathcal{F}}
\newcommand{\GG}{\mathcal{G}}
\newcommand{\HH}{\mathcal{H}}

\newcommand{\LL}{\mathcal{L}}
\newcommand{\KK}{\mathcal{K}}

\newcommand{\UU}{\mathcal{U}}
\newcommand{\TT}{\mathcal{T}}

\newcommand{\D}{\mathbb{D}}
\newcommand{\Isom}{\mathrm{Isom}}
\newcommand{\Ind}{\mathrm{Ind}}
\renewcommand{\AA}{\mathcal{A}}
\newcommand{\EE}{\mathcal{E}}
\newcommand{\CC}{\mathcal{C}}
\renewcommand{\L}{\mathrm{L}}
\newcommand{\MAlg}{\mathrm{MAlg}}
\newcommand{\E}{\mathbb{E}}

\newcommand{\A}{\mathbf{A}}
\newcommand{\B}{\mathbf{B}}
\newcommand{\N}{\mathbb{N}}
\newcommand{\Z}{\mathbb{Z}}
\newcommand{\Q}{\mathbb{Q}}
\newcommand{\U}{\mathbb{U}}
\newcommand{\R}{\mathbb{R}}
\newcommand{\C}{\mathbb{C}}

\newcommand{\M}{\mathbf{M}}

\newcommand{\Sym}{\mathrm{Sym}}

\newcommand{\inv}{^{-1}}

\newcommand{\defin}[1]{\textbf{\textit{#1}}}

\newcommand{\Stab}{\mathrm{Stab}}

\newcommand{\dom}{\mathrm{dom}}
\newcommand{\Aut}{\mathrm{Aut}}

\newcommand{\tp}{\mathrm{tp}}

\newcommand\Nm[2]{{#1}_{\{#2\}}}

\newcommand{\indep}{%
\mathrel{\reflectbox{\rotatebox[origin=c]{90}{$\models$}}}}

\newcommand{\supp}{\mathrm{supp}}

\newcommand{\x}{\bar{x}}
\newcommand{\y}{\bar{y}}

\newcommand{\Addresses}{{
			\bigskip
			\footnotesize
			\noindent R.~Barritault, \textsc{Universite Claude Bernard Lyon 1, CNRS, Centrale Lyon, INSA Lyon, Université Jean Monnet, ICJ UMR5208, 69622
Villeurbanne, France.}\par\nopagebreak\noindent
			\textit{E-mail address: }\texttt{barritault@math.univ-lyon1.fr}

                \medskip

			\noindent C.~Jahel, \textsc{American University of Beirut, \\Beirut, Lebanon}\par\nopagebreak\noindent
			\textit{E-mail address: }\texttt{cj21@aub.edu.lb}
			
			\medskip
			
			\noindent M.~Joseph, \textsc{University of Paris-Cité, Institut de Mathématiques de Jussieu-Paris Rive Gauche,Place Aurélie Nemours, 75013 Paris (France)}\par\nopagebreak\noindent
			\textit{E-mail address: }\texttt{joseph@imj-prg.fr}
	
			\medskip

	}}